\documentclass{amsart}
\pdfoutput=1
\usepackage{graphicx,color,epsf}
\usepackage [cmtip,arrow]{xy}
\xyoption{all}
\usepackage{subfig}
\usepackage{float}

\usepackage {pb-diagram,pb-xy}

\usepackage{amsmath,amssymb}
\newtheorem{theorem}{Theorem}[section]
\newtheorem{lemma}[theorem]{Lemma}
\newtheorem{corollary}[theorem]{Corollary}
\newtheorem{proposition}[theorem]{Proposition}
\newtheorem{notationdefinition}[theorem]{Notation and definition}

\theoremstyle{definition}
\newtheorem{definition}[theorem]{Definition}
\newtheorem{example}[theorem]{Example}

\newtheorem{notation}[theorem]{Notation}

\theoremstyle{remark}
\newtheorem{remark}[theorem]{Remark}

\definecolor{DarkBlue}{rgb}{0,0.1,0.55}

\numberwithin{equation}{section}

\newcommand {\hide}[1]{}

\newcommand {\junk}[1]{}

\newcommand {\R} {{\rm R}}

\newcommand {\Real}[1]   {\mbox{$\mathbb{R}^{#1}$}}
\newcommand {\Sphere}{\mbox{${\bf S}$}}               
 \newcommand {\re}         {\Real{}}

\newcommand {\Z}  {\mathbb{Z}}
 \newcommand {\N}         {\mathbb{N}}

\newcommand {\ZZ} {{\rm Z}}

\newcommand {\V} {\mathbf{V}}

\newcommand {\eps} {{\varepsilon}}
\newcommand {\ep } {{\varepsilon}}

\newcommand {\Id} {\mbox{\rm Id}}

\newcommand {\T}      {{\mbox{\rm T}}}

\newcommand {\Zer} {{\rm Zer}}
\newcommand {\Reali} {{\rm Reali}}

\def\addots{\mathinner{\mkern1mu
\raise1pt\vbox{\kern7pt\hbox{.}}
\mkern2mu\raise4pt\hbox{.}\mkern2mu
\raise7pt\hbox{.}\mkern1mu}}

\newcommand{\defeq}{\;{\stackrel{\text{\tiny def}}{=}}\;}

\newcommand{\X}{\mathbf{X}}

\newcommand{\TT}{\mathbb{T}}

\newcommand{\limit}{\mathrm{limit}}
\newcommand{\card}{\mathrm{card}}

\newcommand{\zz}{\begin{flushright}$\Box$\end{flushright}}

\newcommand{\halfspace}
    {\quad \hspace{-0.12 in}  \hspace{-0.11 in} \quad }
\newcommand{\halfhalfspace}
    {\quad \hspace{-0.14 in}  \hspace{-0.15 in} \quad }

\newcommand{\halfbackspace}
    {\hspace{-0.06 in}}

\newcommand{\fraction}[2]
    {\textstyle{\frac{#1}{#2}}}

\begin{document}
\title[]
{On homotopy types of limits of semi-algebraic sets
and additive complexity of polynomials
}
\author{Sal Barone}
\address{Department of Mathematics,
Purdue University, West Lafayette, IN 47906, U.S.A.}
\email{sbarone@math.purdue.edu}
\author{Saugata Basu}
\address{Department of Mathematics,
Purdue University, West Lafayette, IN 47906, U.S.A.}
\email{sbasu@math.purdue.edu}

\thanks{The
authors were supported in part by an NSF
grant CCF-0634907.
}

\date{\textbf{\today}}

\maketitle
\begin{abstract}
We prove that the number of
distinct
homotopy types of limits of one-parameter
semi-algebraic families of closed and bounded semi-algebraic sets is
bounded singly exponentially in the additive complexity of any
quantifier-free first order formula defining the family.  As an
important consequence, we derive that the number of
distinct
homotopy types of
semi-algebraic subsets of $\mathbb{R}^k$ defined by a quantifier-free first
order formula $\Phi$, where the sum of the additive complexities of
the polynomials appearing in $\Phi$ is at most $a$, is bounded by
$2^{(k+a)^{O(1)}}$. 
This proves a conjecture made in \cite{BV06}.
\end{abstract}

\newcommand{\constM}{(p+1)(k+a+2)+2k\binom{p+1}{2}}
\newcommand{\constMprime}{(p+1)(s+2)+3\binom{p+1}{2}+3}

\section{Introduction and statement of the main results}
If $S$ is a semi-algebraic subset of $\mathbb{R}^k$ defined by a
quantifier-free first order formula $\Phi$,
then various topological invariants of $S$ (such as the Betti
numbers) can be bounded in terms of the ``format'' of the
formula $\Phi$ (we define format of a formula more precisely below).
The first results in this direction were proved by Ole{\u\i}nik
and Petrovski{\u\i}
\cite{O,OP} (also independently by
Thom \cite{T}, and Milnor \cite{Milnor2})
who proved singly exponential bounds on the Betti numbers of real algebraic
varieties in $\mathbb{R}^k$  defined by polynomials of degree bounded by $d$. These
results were extended to more general semi-algebraic sets in
\cite{B99,GaV,GVZ04,GV07}.
As a consequence of more general finiteness results of
Pfaffian functions, Khovanski\u{\i} \cite{Kho} proved singly exponential
bounds on the number of connected components of real algebraic varieties
defined by polynomials with a fixed  number of monomials. We refer the
reader to the survey article \cite{BPR10} for a more detailed survey
of results on bounding the Betti numbers of semi-algebraic sets.

A second type of quantitative results on the topology of semi-algebraic
sets, more directly relevant to the current paper, seeks
to obtain tight bounds on the
number of different topological types of semi-algebraic sets definable by
first order formulas of bounded format.
If the format of a first-order
formula is specified by the number and
degrees of the polynomials appearing in it (this is often called the
``dense format'' in the  literature), then
it follows from the well-known Hardt's triviality theorem
for semi-algebraic sets
(see \cite{Hardt, BCR})
that this number is finite. However,
the quantitative bounds on the number of topological types
that follow from the proof of Hardt's theorem
are doubly  exponential (unlike the singly exponential bounds on the
Betti numbers).
For some other notions of format,
the finiteness of topological types while being true is not
an immediate consequence of Hardt's theorem (see below), and tight
quantitative bounds on the number of topological types are lacking.

If instead of homeomorphism types, one considers the weaker notion of
\emph{homotopy types}, then
singly exponential bounds have been obtained on the number
of
distinct
\emph{homotopy types} of  semi-algebraic sets defined
by different classes of formulas of bounded format \cite{BV06, Basu10}.

The main motivation behind this paper
is to obtain a singly exponential bound on the number of distinct  homotopy
types of  semi-algebraic sets defined by polynomials of bounded
``additive complexity'' (defined below) answering a question posed in
\cite{BV06}.

One notion of format that will play an important role in this paper
is that of
``additive complexity''.
Roughly speaking the additive complexity of a
polynomial (see Definition \ref{def:rational_additive} below for a
precise definition) is bounded from above by the number of additions
in any straight line program (allowing divisions) that computes the
values of the polynomial at generic points of $\mathbb{R}^n$.
This measure of complexity strictly generalizes the more familiar
measure of complexity of real polynomials
based on counting the number of monomials in the support
(as in Khovanski\u{\i}'s  theory of ``Fewnomials'' \cite{Kho}),
and is thus of considerable interest in quantitative real
algebraic geometry.
Additive complexity of real univariate polynomials was first considered
in the context of computational complexity theory by Borodin and Cook
\cite{Borodin-Cook76}, who proved an effective bound on the number
of real zeros of an univariate
polynomial in terms of its additive
complexity. This result was further improved upon by
Grigoriev \cite{Grigoriev82}
and Risler \cite{Risler85} who applied
Khovanski\u{\i}'s  results on fewnomials \cite{Kho}.
A surprising
fact conjectured in \cite{BRbook}, and proved by Coste \cite{CosteFew}
and van den Dries \cite{Dries}, is that the number of topological
types of real algebraic varieties defined by polynomials of bounded
additive complexity is finite.

\subsection{Bounding the number of homotopy types of semi-algebraic sets}
The problem of obtaining tight quantitative bounds on the topological
types of semi-algebraic sets defined by formulas of bounded format was
considered in \cite{BV06}.  Several results (with different notions of
formats of formulas) were proved in \cite{BV06}, each giving an
explicit singly exponential (in the number of variables and size of
the format) bound on the number of
distinct
homotopy types of semi-algebraic
subsets of $\mathbb{R}^k$ defined by formulas having format of bounded size.
However, the case of additive complexity was left open in \cite{BV06},
and only a strictly weaker result was proved in the case of
\emph{division-free} additive complexity.
\footnote{Note that what we call
  ``additive complexity'' is called ``rational additive complexity''
  in \cite{BV06}, and what we call ``division-free additive
  complexity'' is called ``additive complexity'' there.}
In order to state this result
precisely, we need a few preliminary definitions.

\begin{definition}
\label{def:additive}
The \emph{division-free additive complexity} of a polynomial is a
non-negative integer,
and we say that
a polynomial $P \in \mathbb{R}[X_1, \ldots ,X_k]$ has \emph{division-free additive
complexity
at most $a$},
$a\geq 0$, if there are polynomials $Q_1, \ldots , Q_a \in \mathbb{R}[X_1,
  \ldots ,X_k]$ such that
\begin{itemize}
\item[(i)]
$Q_1=u_1X_{1}^{\alpha_{11}} \cdots X_{k}^{\alpha_{1k}} +
v_1X_{1}^{\beta_{11}} \cdots X_{k}^{\beta_{1
k}}$,\\
where $u_1, v_1 \in \mathbb{R}$, and
$\alpha_{11}, \ldots ,\alpha_{1k}, \beta_{11}, \ldots , \beta_{1k} \in \N$;

\item[(ii)]
$Q_j=u_jX_{1}^{\alpha_{j1}} \cdots X_{k}^{\alpha_{j k}}
\prod_{1 \le i \le j-1}Q_{i}^{\gamma_{j i}} +
v_jX_{1}^{\beta_{j1}} \cdots X_{k}^{\beta_{j k}}\prod_{1 \le i \le j-1}Q_{i}^{\delta_{ji}}$,\\
where $1 < j \le a$, $u_j, v_j \in \mathbb{R}$, and
$\alpha_{j1}, \ldots ,\alpha_{j k}, \beta_{j1}, \ldots , \beta_{j k},
\gamma_{ji}, \delta_{ji} \in \N$ for $1 \le i <j$;

\item[(iii)]
$P= cX_{1}^{\zeta_{1}} \cdots X_{k}^{\zeta_{k}}\prod_{1 \le j \le a}Q_{j}^{\eta_{j}}$,\\
where $c \in \mathbb{R}$, and $\zeta_1, \ldots , \zeta_k, \eta_1, \ldots ,\eta_a \in \N$.
\end{itemize}
In this case, we say that the above sequence of equations
is a \emph{division-free additive representation} of $P$ of length $a$.
\end{definition}

In other words, $P$ has division-free additive complexity at most $a$ if
there exists a straight line program which,
starting with variables
$X_1, \ldots ,X_m$ and constants in $\mathbb{R}$ and
applying additions and multiplications,
computes $P$ and
which uses at most $a$ additions
(there is no bound on the number of multiplications).
Note that
the additive complexity of a polynomial (cf. Definition \ref{def:rational_additive})
is clearly at most its
division-free additive complexity, but can be much smaller (see
Example \ref{ex:rational_additive} below).

\begin{example}
\label{ex:additive}
The polynomial $P:=(X+1)^d \in \mathbb{R}[X]$ with $0<d \in \Z$, has $d+1$ monomials when expanded but division-free additive complexity at most 1.
\end{example}

\begin{notation}
We denote by
$\mathcal{A}^{\mathrm{div-free}}_{k,a}$ the family of ordered (finite) lists
${\mathcal P}=(P_1, \ldots , P_s)$ of polynomials $P_i \in \mathbb{R}[X_1,\ldots,X_k]$,
with the division-free additive complexity of every $P_i$ not exceeding $a_i$, with
$a=\sum_{1 \le i \le s}a_i$.
Note that $\mathcal{A}^{\mathrm{div-free}}_{k,a}$ is allowed to contain lists
of different sizes.
\end{notation}

Suppose that $\phi$ is a Boolean formula with atoms
$\{p_i,q_i,r_i \mid 1 \leq i \leq s\}$.
For an ordered list ${\mathcal P} = (P_1,\ldots,P_s)$ of polynomials
$P_i \in {\mathbb{R}[X_1,\ldots,X_k]}$, we denote by $\phi_{\mathcal P}$
the formula obtained from $\phi$ by replacing
for each $i,\ 1\leq i \leq s $, the atom $p_i$
(respectively, $q_i$ and $r_i$) by
$P_i= 0$ (respectively, by $P_i > 0$ and by $P_i < 0$).

\begin{definition}\label{def:hom_lists}
We say that two ordered lists ${\mathcal P} = (P_1,\ldots,P_s)$,
${\mathcal Q} = (Q_1,\ldots,Q_s)$
of polynomials $P_i, Q_i \in {\mathbb{R}[X_1,\ldots,X_k]}$ have the same \emph{homotopy
type} if for any Boolean formula $\phi$, the semi-algebraic sets defined by
$\phi_{\mathcal P}$ and $\phi_{\mathcal Q}$ are
homotopy equivalent.
Clearly, in order to be homotopy equivalent two lists should have equal size.
\end{definition}

\begin{example}
\label{eg:homotopy_types}
Consider the lists
$\mathcal{P} = (X_1,X_2^2,X_1^2+X_2^2+1)$
and $\mathcal{Q}= (X_1^3, X_2^4,1)$. It is easy to see that they
have the same homotopy type, since in this case
for
each Boolean formula
$\phi$ with $9$ atoms, the semi-algebraic sets
defined by $\phi_{\mathcal{P}}$ and $\phi_{\mathcal{Q}}$ are
identical.
A slightly more non-trivial example is provided by
$\mathcal{P} = (X_2 - X_1^2, X_2)$ and $\mathcal{Q} = (X_2, X_2 + X_1^2)$.
In this case,
for
each Boolean formula
$\phi$ with $6$ atoms, the semi-algebraic sets
defined by $\phi_{\mathcal{P}}$ and $\phi_{\mathcal{Q}}$ are
not identical but
homeomorphic. Finally, the singleton sequences
$\mathcal{P} = (X_2X_1(X_1-1))$ and $\mathcal{Q} = (X_2(X_1^2 - X_2^4))$
are homotopy equivalent. In this case the semi-algebraic sets
sets
defined by $\phi_{\mathcal{P}}$ and $\phi_{\mathcal{Q}}$ are
homotopy equivalent, but not necessarily
homeomorphic. For instance, the algebraic set
defined by $X_2X_1(X_1-1)=0$ is
homotopy equivalent to
the algebraic set defined by $X_2(X_1^2 - X_2^4)=0$, but they are not
homeomorphic to each other.
\end{example}

The following theorem is proved in \cite{BV06}.

\begin{theorem}\cite{BV06}
\label{thm:additive}
The number of
distinct
homotopy types of ordered lists in
$\mathcal{A}^{\mathrm{div-free}}_{k,a}$
does not exceed
\begin{equation}\label{eq:additive}
2^{O(k+a)^8}.
\end{equation}
In particular, if $\phi$ is any Boolean formula with $3s$ atoms,
the number of
distinct
homotopy types of the  semi-algebraic sets defined
by $\phi_{\mathcal P}$, where
${\mathcal P} = (P_1,\ldots,P_s)\in \mathcal{A}^{\mathrm{div-free}}_{k,a}$,
does not exceed (\ref{eq:additive}).
\end{theorem}

\begin{remark}
The bound in \ref{eq:additive} in Theorem \ref{thm:additive}
is stated in a slightly different form than in the original paper, to take
into account the fact that by our definition the division-free
additive complexity of a polynomial (for example, that of a monomial)
is allowed to be $0$. This is not an important issue (see Remark
\ref{rem:zero} below).
\end{remark}

The additive complexity of a polynomial is defined
as follows \cite{Borodin-Cook76, Grigoriev82, Risler85, BRbook}.

\begin{definition}
\label{def:rational_additive}
A polynomial $P \in \mathbb{R}[X_1, \ldots ,X_k]$ is said to have
\emph{additive complexity}
at most $a$ if there are \emph{rational functions}
$Q_1, \ldots , Q_a \in \mathbb{R}(X_1, \ldots ,X_k)$ satisfying conditions
(i), (ii), and (iii)
in Definition \ref{def:additive}
with $\N$ replaced by $\Z$.
In this case we say that the above sequence of equations
is an additive representation of $P$ of length $a$.
\end{definition}

\begin{example}
\label{ex:rational_additive}
The polynomial $X^d+ \cdots + X+1 =(X^{d+1}-1)/(X-1)\in \mathbb{R}[X]$
with $0<d \in \Z$, has  additive complexity
(but not division-free additive complexity) at most $2$ (independent
of $d$).
\end{example}

\begin{notation}
We denote by
$\mathcal{A}_{k,a}$ the family of ordered (finite) lists
${\mathcal P}=(P_1, \ldots , P_s)$ of polynomials $P_i \in \re[X_1,\ldots,X_k]$,
with the  additive complexity of every $P_i$ not exceeding $a_i$,
with
$a=\sum_{1 \le i \le s}a_i$.
\end{notation}

It was conjectured in \cite{BV06} that
Theorem \ref{thm:additive} could be strengthened by replacing
$\mathcal{A}^{\mathrm{div-free}}_{k,a}$ by $\mathcal{A}_{k,a}$.
In this paper we prove this conjecture. More formally, we prove

\begin{theorem}
\label{thm:main1}
The number of
distinct
homotopy types of ordered lists in
$\mathcal{A}_{k,a}$
does not exceed
$2^{(k+a)^{O(1)}}$. 
\end{theorem}

\subsection{Additive complexity and limits of semi-algebraic sets}
The proof of Theorem \ref{thm:additive} in \cite{BV06} proceeds by reducing the
problem to the case of bounding the number of
distinct
homotopy types of
semi-algebraic sets defined by polynomials having a bounded number
of monomials.
The reduction
which was already used by Grigoriev \cite{Grigoriev82} and Risler \cite{Risler85}
is as follows.
Let  ${\mathcal P} \in \mathcal{A}^{\mathrm{div-free}}_{k,a}$ be an ordered list.
For each polynomial $P_i \in {\mathcal P}$, $1 \le i \le s$,
consider the sequence of polynomials
$Q_{i 1}, \ldots, Q_{i a_i}$ as in Definition~\ref{def:additive}, so that
$$
P_i:=c_i X_{1}^{\zeta_{i 1}} \cdots X_{k}^{\zeta_{i k}}\prod_{1 \le j \le a_i}Q_{i j}^{\eta_{i j}}.
$$
Introduce $a_i$ new variables
$Y_{i1}, \ldots ,Y_{i a_i}$.
Fix a semi-algebraic set $S \subset \mathbb{R}^m$,
defined by a formula $\phi_{\mathcal P}$.
Consider
the semi-algebraic set $\widehat S$, defined by the conjunction of $a$
3-nomial equations obtained from equalities in (i), (ii) of
Definition~\ref{def:additive} by replacing
$Q_{i j}$ by $Y_{i j}$
for all $1 \le i \le s$, $1 \le j \le a_k$, and the formula $\phi_{\mathcal P}$
in which every occurrence of an atomic formula
of the kind $P_k \ast 0$, where $\ast \in \{ =, >, < \}$,
is replaced by the formula
$$
c_i X_{1}^{\zeta_{i 1}} \cdots X_{k}^{\zeta_{i k}}\prod_{1 \le j \le a_i}Y_{i j}^{\eta_{i j}}
\ast 0.
$$
Note that $\widehat S$ is a
semi-algebraic subset of $\mathbb{R}^{k+a}$.

Let $\rho:\> \mathbb{R}^{k+a} \to \mathbb{R}^k$
be the projection map on the subspace
spanned by
$X_1, \ldots ,X_k$.
It is clear that the restriction $\rho_{\widehat S}:\> \widehat S \to S$ is a
homeomorphism, and moreover $\widehat S$ is defined by polynomials having
at most $k+a$ monomials.
Thus, in order to bound the number of
distinct
homotopy types
for $S$, it suffices to bound the same number for $\widehat{S}$, but
since   $\widehat S$ is defined by at most $2 a$ polynomials in
$k+a$ variables having at most $k+a$ monomials in total,
we have reduced the problem of bounding the number
of distinct homotopy types occurring in
$\mathcal{A}^{\mathrm{div-free}}_{k,a}$,
to that of bounding the the number
of distinct homotopy types of semi-algebraic sets
defined by at most $2 a$ polynomials
in $k+a$ variables,
with the total number of monomials appearing bounded by $k+a$.
This allows us to apply a bound proved in the fewnomial case in \cite{BV06},
to obtain
a singly exponential bound on the number of distinct homotopy types
occurring in  $\mathcal{A}^{\mathrm{div-free}}_{k,a}$.

Notice that for the map $\rho_{\widehat S}$ to be a homeomorphism it is
crucial that the exponents $\eta_{i j},\gamma_{i j},\delta_{i j}$
be non-negative, and this restricts
the proof to the case of division-free additive complexity.
We overcome this difficulty as follows.

Given a polynomial $F \in \mathbb{R}[X_1,\ldots,X_k]$ with
additive complexity bounded by $a$, we
prove that $F$ can be expressed as a quotient $\frac{P}{Q}$ with
$P,Q \in \mathbb{R}[X_1,\ldots,X_k]$ with
the sum of the \emph{division-free} additive complexities of $P$ and $Q$ bounded by $a$
(see Lemma \ref{lem:equivalence} below).
We then express the set of real zeros of $F$ in $\mathbb{R}^k$ inside any
fixed closed ball
as the Hausdorff limit of a
one-parameter semi-algebraic family defined using the polynomials
$P$ and $Q$ (see
Proposition \ref{prop:sectionfivemain} and the accompanying
Example \ref{eg:main} below).

While the limits of one-parameter semi-algebraic families defined by
polynomials with bounded division-free additive complexities
themselves can have complicated descriptions which cannot be described
by polynomials of bounded division-free additive complexity, the
topological complexity (for example, measured by their Betti numbers)
of such limit sets are well controlled.  Indeed, the problem of
bounding the Betti numbers of Hausdorff limits of one-parameter
families of semi-algebraic sets was considered by Zell in
\cite{Hausdorff}, who proved a singly exponential bound on the Betti
numbers of such sets.  We prove in this paper (see
Theorems \ref{thm:main_weak} and  \ref{thm:main} below)
that the number of
distinct
homotopy types of such limits
can indeed be bounded singly exponentially in terms of the format of
the formulas defining the one-parameter family.  The techniques
introduced by Zell in \cite{Hausdorff} (as well certain semi-algebraic
constructions described in \cite{BZ09}) play a crucial role in the
proof of our bound.  These intermediate results may be of independent
interest.

Finally, applying
Theorem \ref{thm:main_weak}
to the one-parameter
family referred to in the previous paragraph, we obtain a bound
on the number of
distinct
homotopy types of real algebraic varieties defined by
polynomials having bounded additive complexity. The semi-algebraic
case requires certain additional techniques and is dealt with in
Section \ref{subsec:semi-algebraic}.

\subsection{Homotopy types of limits of semi-algebraic sets}
In order to state our results on bounding the number of
distinct
homotopy types of limits
of one-parameter families of semi-algebraic sets we need to introduce
some notation.

\begin{notation}
For any first order formula $\Phi$
with $k$ free variables, if $\mathcal{P}\subset \mathbb{R}[X_1,\dots,X_k]$
consists of the polynomials appearing in $\Phi$, then we call $\Phi$ a $\mathcal{P}$-formula.
\end{notation}

\begin{notation}[Format of first-order formulas]
\label{notation:format}
Suppose $\Phi$ is a
$\mathcal{P}$-formula defining a semi-algebraic subset of $\mathbb{R}^k$ involving  $s$
  polynomials of degree at most $d$.  In this case we say that $\Phi$
  has \emph{dense format} $(s,d,k)$.
If $\mathcal{P}\in \mathcal{A}_{k,a}$
then we say that $\Phi$ has
\emph{additive format bounded by} $(a,k)$.
If $\mathcal{P}\in \mathcal{A}^{\mathrm{div-free}}_{k,a}$ then
we say that $\Phi$ has \emph{division-free additive format bounded by} $(a,k)$.
\end{notation}

\begin{remark}{\label{rem:zero}}
A monomial has additive complexity
0, but every $\mathcal{P}$-formula with $\mathcal{P}\subset \mathbb{R}[X_1,\ldots,X_k]$ containing only monomials is equivalent to a $\mathcal{P}'$-formula, where $\mathcal{P}'=\{X_1,\ldots,X_k\}$.
In particular, if $\phi$ is a $\mathcal{P}$-formula with
(division-free) additive format bounded by $(a,k)$, then $\phi$ is equivalent to a $\mathcal{P}'$-formula having
(division-free)
additive format bounded by $(a,k)$ and such that the cardinality of $\mathcal{P}'$ is
at most $a+k$.
\end{remark}

\begin{notation}{\label{not:limit}}
For
any $k \geq 1$, and
 $ 1\leq p \leq q \leq k$,
we denote by  $\pi_{[p,q]}: \mathbb{R}^k=\mathbb{R}^{[1,k]}\rightarrow \mathbb{R}^{[p,q]}$ the projection
$$(x_1, \ldots, x_k)\mapsto (x_p, \ldots, x_q)$$
(omitting the dependence on $k$ which should be clear from context).
In case $p=q$ we
will denote by $\pi_p$ the projection $\pi_{[p,p]}$. %
For any semi-algebraic subset
$X\subset \mathbb{R}^{k+1}$,
 and $\lambda\in \mathbb{R}$,
we denote by $X_\lambda$ the following semi-algebraic subset of $\mathbb{R}^k$:
$$X_\lambda = \pi_{[1,k]}(X\cap \pi_{k+1}^{-1}(\lambda)).$$
We denote by $\mathbb{R}_+$ the set of strictly positive elements of $\mathbb{R}$.
If additionally $X\subset \mathbb{R}^k\times \mathbb{R}_+$, then we denote by $X_{\limit}$
the following semi-algebraic subset of $\mathbb{R}^k$:
$$ X_{\limit} := \pi_{[1,k]}(\overline{X} \cap \pi_{k+1}^{-1}(0)),$$
where $\overline{X}$ denotes the topological closure of $X$ in
$\mathbb{R}^{k+1}$.
\zz
\end{notation}

We have the following theorem
which establishes a singly exponential bound on
the number of
distinct
homotopy types of the Hausdorff limit
of a one-parameter family of compact semi-algebraic sets defined by a
first-order formula of bounded additive format. This result complements
the result in \cite{BV06} giving singly exponential bounds on the
homotopy types of semi-algebraic sets defined by first-order
formulas having bounded
division-free
additive format on one hand, and the
result of Zell \cite{Hausdorff} bounding the Betti numbers
of the Hausdorff
limits of one-parameter families of semi-algebraic sets on the other,
and could be of independent interest.

\begin{theorem}
\label{thm:main}
For each $a,k\in \mathbb{N}$,
there exists a finite collection
$\mathcal{S}_{k,a}$
of semi-algebraic subsets of
$\mathbb{R}^N$, $N=(k+2)(k+1)+\binom{k+2}{2}$, with
$\card\; \mathcal{S}_{k,a}=2^{(k+a)^{O(1)}}$,
which satisfies
the following property.
If $\TT\subset \mathbb{R}^k\times \mathbb{R}_+$
is a bounded semi-algebraic set described by a formula having
additive format bounded by $(a,k+1)$
such that $\TT_t$ is closed for each $t>0$,
then
$\TT_{\limit}$
is
homotopy equivalent to some
$S\in \mathcal{S}_{k,a}$
(cf. Notation \ref{not:limit}).
\end{theorem}

The rest of the paper is devoted to the proofs of Theorems \ref{thm:main}
and \ref{thm:main1} and is organized as follows. We first prove
a weak version (Theorem \ref{thm:main_weak}) of
Theorem \ref{thm:main} in Section \ref{sec:main}, in which the term
``additive complexity'' in the statement of Theorem \ref{thm:main}
is replaced by the term
``division-free additive complexity''.
Theorem \ref{thm:main_weak} is
then used  in Section \ref{sec:main1}
to prove Theorem \ref{thm:main1} after introducing some additional
techniques, which in turn is
used to prove Theorem \ref{thm:main}.

\section
{Proof of a weak version of Theorem \ref{thm:main}}
\label{sec:main}

In this section we prove the following weak version of Theorem \ref{thm:main}
(using \emph{division-free} additive format rather than additive format) which
is needed in the proof of Theorem \ref{thm:main1}.

\begin{theorem}
\label{thm:main_weak}
For each $a,k\in \mathbb{N}$,
there exists a finite collection
$\mathcal{S}_{k,a}$
of semi-algebraic subsets of
$\mathbb{R}^N$, $N=(k+2)(k+1)+\binom{k+2}{2}$, with
$\card\; \mathcal{S}_{k,a}= 2^{O(k(k^2+a))^{8}} = 
2^{(k+a)^{O(1)}}
$,
which satisfies
the following property.
If
$\TT\subset \mathbb{R}^k\times \mathbb{R}_+$
is a bounded semi-algebraic set described by a formula having
\emph{division-free}
additive format bounded by $(a,k+1)$
such that $\TT_t$ is closed for each $t>0$,
then
$\TT_{\limit}$
is
homotopy equivalent to some $S\in \mathcal{S}_{k,a}$
(cf. Notation \ref{not:limit}).
\end{theorem}

\subsection{Outline of the proof}
The main steps in the
proof of Theorem \ref{thm:main_weak}
are as follows.  Let $\TT \subset
\mathbb{R}^{k}\times
\mathbb{R}_{+}$
be a bounded semi-algebraic set,
such that
$\TT_t$ is closed for each $t \in \mathbb{R}$,
and let
$\TT_{\limit}$
be as in Notation \ref{not:limit}.

We first prove that for all  small enough $\lambda>0$,
there exists a semi-algebraic surjection
$f_\lambda:\TT_\lambda \rightarrow \TT_{\limit}$
which is metrically close to the identity map $1_{\TT_\lambda}$
(see Proposition \ref{prop:12} below).
Using a semi-algebraic realization of the fibered join described in
\cite{BZ09}
(see also \cite{GVZ04}),
we then
consider, for any fixed $p \geq 0$, a semi-algebraic set
$\mathcal{J}^p_{f_\lambda}(\TT_\lambda)$ which is $p$-equivalent to
$\TT_\limit$
(see Proposition \ref{prop:5}). The
definition of $\mathcal{J}^p_{f_\lambda}(\TT_\lambda)$ still involves the map
$f_\lambda$, whose definition is not simple, and hence we cannot
control the topological type of $\mathcal{J}^p_{f_\lambda}(\TT_\lambda)$
directly.
However, the fact that $f_\lambda$ is
metrically close to the identity map
enables us to adapt the main technique in \cite{Hausdorff} due to Zell. We
replace $\mathcal{J}^p_{f_\lambda}(\TT_\lambda)$ by another semi-algebraic set,
which we denote by
$\mathcal{D}^p_\eps(\TT)$ (for $\eps>0$ small enough),
which is  homotopy equivalent to
$\mathcal{J}^p_{f_\lambda}(\TT_\lambda)$, but
whose definition no longer involves the
map $f_\lambda$ (Definition \ref{defncalD}).
We can now bound the format of $\mathcal{D}^p_\eps(\TT)$
in terms of the format of the formula defining $\TT$.
This key result is summarized in Proposition  \ref{prop:main}.

We first recall the definition of $p$-equivalence
(see, for example, \cite[page 144]{tomDieck08}).

\begin{definition}[$p$-equivalence]
\label{def:p-equivalence}
A map $f: A \rightarrow B$ between two topological spaces is called a
\emph{$p$-equivalence} if the induced
map
\[
f_*: \boldsymbol{\pi}_i(A,a) \rightarrow \boldsymbol{\pi}_i(B,f(a))
\]
is,
 for each $a \in A$,
bijective
for  $0 \leq i < p$, and
surjective
for $i=p$,
and we say that $A$ is \em{$p$-equivalent} to $B$.
\end{definition}

\begin{proposition}
\label{prop:main}
Let
$\TT\subset \mathbb{R}^k\times \mathbb{R}_+$
be a bounded
semi-algebraic set such that $\TT_t$ is closed
for each  $t>0$, and let
$p \geq 0$.
Suppose also that $\TT$ is described by a formula
having
(division-free)
additive
format
bounded by
$(a,k+1)$
and dense format
$(s,d,k+1)$.
Then, there exists a semi-algebraic set
$\mathcal{D}^p\subset \mathbb{R}^N$,
$N=(p+1)(k+1)+\binom{p+1}{2}$,
 such  that $\mathcal{D}^p$ is
$p$-equivalent to
$\TT_{\limit}$ (cf. Notation \ref{not:limit})
and such that $\mathcal{D}^p$ is described by a
formula having
(division-free)
additive
 format
bounded by $(M,N)$
and dense format
$(M',d+1,N)$,
where
$M=\constM$
and $M'=\constMprime$.
\end{proposition}

Finally, Theorem \ref{thm:main_weak} is an easy consequence of
Proposition \ref{prop:main}.

\subsection{Preliminaries}
We need a few facts from the homotopy theory of
finite CW-complexes.

We first prove a basic result about $p$-equivalences (Definition \ref{def:p-equivalence}).
It is clear that $p$-equivalence is not an equivalence relation
(e.g., for any $p\geq 0$,
the map taking $\Sphere^p$ to a point is a $p$-equivalence,
but no map from a point into $\Sphere^p$ is one).
However, we have the following.

\begin{proposition}
\label{prop:top_basic}
Let $A,B,C$ be
finite CW-complexes
with $\dim(A),\dim(B) \leq k$ and
suppose that $C$ is $p$-equivalent to $A$ and $B$ for some $p >k$.
Then, $A$ and $B$ are homotopy equivalent.
\end{proposition}

The proof of Proposition \ref{prop:top_basic} will rely on the
following well-known lemmas.

\begin{lemma}\cite[page 182, Theorem 7.16]{Whitehead}
\label{lem:top_basic1}
Let $X,Y$ be CW-complexes  and $f:X\rightarrow Y$ a $p$-equivalence. Then,
for each CW-complex $M$, $\dim(M) \leq p$,
the induced map
\[
f_*: [M,X] \rightarrow [M,Y]
\]
is surjective.
\end{lemma}

\begin{lemma}\cite[page 69]{Viro-homotopy}
\label{lem:top_basic2}
If $A$ and $B$ are finite
CW-complexes, with $dim(A) < p$ and $dim(B) \leq p$, then
every $p$-equivalence from $A$ to $B$ is a homotopy equivalence.
\end{lemma}

\proof[Proof of Proposition \ref{prop:top_basic}]
Suppose $f:C\rightarrow A$ and $g:C \rightarrow B$ are two $p$-equivalences.
Applying  Lemma \ref{lem:top_basic1} with $X=C$, $M=Y=A$, we have
that the homotopy class of the identity map
$1_A$ has a preimage, $[h]$,  under $f_*$,
for some $h \in [A,C]$.
Then, for each $ a \in A$, and $i \geq 0$,
\[
f_*\circ h_* : \boldsymbol{\pi}_i(A,a) \rightarrow \boldsymbol{\pi}_i(A, f\circ h(a)),
\]
is bijective.
In particular, since $f$ is a $p$-equivalence,
this implies that $h_*: \boldsymbol{\pi}_i(A,a) \rightarrow \boldsymbol{\pi}_i(C,h(a))$ is
bijective
for $0 \leq i < p$. Composing $h$ with $g$, and noting that $g$ is
also a $p$-equivalence we get that the map
$(g \circ h)_*:\boldsymbol{\pi}_i(A,a) \rightarrow
\boldsymbol{\pi}_i(B, g\circ h (a))$ is bijective
for $0 \leq i < p$.
Now, applying Lemma \ref{lem:top_basic2} we get that $g \circ h$ is a
homotopy equivalence.
\zz

We introduce some more notation.

\begin{notation}
For any
$R\in \mathbb{R}_+$,
we denote by $B_k(0,R) \subset \mathbb{R}^k$,
the open ball of radius $R$ centered at the origin.
\end{notation}

\begin{notation}
For
$P \in \mathbb{R}[X_1,\ldots,X_k]$,
we denote by $\Zer(P,\mathbb{R}^k)$ the real
algebraic set defined by $P=0$.
\end{notation}

\begin{notation}{\label{not:reali}}
For any first order formula $\Phi$
with $k$ free variables, we denote by $\Reali(\Phi)$ the semi-algebraic
subset of $\mathbb{R}^k$ defined by $\Phi$.
\end{notation}

A very important construction that we use later in the paper is an efficient
semi-algebraic realization (up to homotopy) of the iterated
fibered join of a semi-algebraic set over a semi-algebraic map.
This construction was  introduced in \cite{BZ09}.

\subsection{Topological definitions}
We first recall the basic definition of the
the iterated join of a topological space.

\begin{notation}
\label{not:simplex}
For each
$p \geq 0$,
we denote
\[
\Delta_{[0,p]} = \{\textbf{t} = (t_0,\ldots,t_p)\mid t_i \geq  0, 0 \leq i \leq p, \sum_{i=0}^p
t_i = 1,
\}
\]
the standard $p$-simplex.
For each subset $I = \{i_0,\ldots,i_m\}, 0 \leq i_0 < \cdots < i_m \leq p$,
let $\Delta_I \subset \Delta_{[0,p]}$ denote the face
$$\Delta_I = \{ \textbf{t} = (t_0,\ldots,t_p)
\in \Delta_{[0,p]}\; \mid \; t_i = 0 \mbox{ for all } i \not\in I \}
$$
of $\Delta_{[0,p]}$.
\end{notation}

\begin{definition}
\label{def:pfoldjoin}
For
$p \geq 0$,
the $(p+1)$-fold join
$J^p(X)$ of a topological space
$X$ is
\begin{equation}
\label{eqn:definitionofjoin}
J^p(X) \defeq \underbrace{X\times\cdots\times X}_{(p+1)\mbox{ times }}
\times \Delta_{[0,p]}/\sim,
\end{equation}
where
\[
(x_0,\ldots,x_p,t_0,\ldots,t_p) \sim (x_0',\ldots,x_p',t_0,\ldots,t_p)
\]
if for each $i$ with $t_i \neq 0$, $x_i = x_i'$.
\end{definition}

In the special situation
when $X$ is a semi-algebraic set,
the space $J^p(X)$ defined above is not immediately
a semi-algebraic set, because of taking quotients.
We now define a semi-algebraic set,
$\mathcal{J}^p(X)$, that is
homotopy equivalent
to $J^p(X)$.

Let $\Delta'_{[0,p]} \subset \mathbb{R}^{p+1}$ denote the set defined by
$$
\Delta'_{[0,p]} = \{ \textbf{t} = (t_0,\ldots,t_p) \in \mathbb{R}^{p+1} \;\mid\; \sum_{0 \leq i \leq p} t_i = 1 , |\textbf{t}|^2 \leq 
1\}.
$$

For each subset $I = \{i_0,\ldots,i_m\}, 0 \leq i_0 < \cdots < i_m \leq p$,
let $\Delta'_I \subset \Delta'_{[0,p]}$ denote
$$
\Delta'_I = \{ \textbf{t} = (t_0,\ldots,t_p)
\in \Delta'_{[0,p]}\; \mid \; t_i = 0 \mbox{ for all } i \not\in I \}.
$$

It is clear that the standard simplex
$\Delta_{[0,p]}$
is a deformation retract of $\Delta'_{[0,p]}$
via a deformation retraction,
$\rho_p: \Delta'_{[0,p]} \rightarrow \Delta_{[0,p]}$,
that restricts to a deformation retraction
$\rho_p|_{\Delta'_I}:\Delta'_I\rightarrow \Delta_I$ for each $I \subset [0,p]$.

We use the lower case bold-face notation $\mathcal{\textbf{x}}$ to denote a
point
$\mathcal{\textbf{x}}=(x_1,\dots,x_k)$ of $\mathbb{R}^k$
and upper-case $\X=(X_1,\dots,X_k)$
to denote a \emph{block of variables}.
In the following definition the role of the
$\binom{p+1}{2}$
variables
$(A_{ij})_{0 \leq i < j \leq p}$ can be safely ignored, since they are
all set to $0$. Their significance will be clear later.

\begin{definition}[The semi-algebraic join \cite{BZ09}%
    \label{def:semi-algebraic-join}]
For a semi-algebraic subset $X \subset \mathbb{R}^k$
contained in $B_k(0,R)$,
defined by a $\mathcal{P}$-formula
$\Phi$,
we define
$$
\begin{aligned}
\mathcal{J}^p(X)=&\{(\mathcal{\textbf{x}}^0,\dots,\mathcal{\textbf{x}}^p,
  \textbf{t},\textbf{a})\in  \mathbb{R}^{(p+1)(k+1)+\binom{p+1}{2}}|\\ &\quad
  \Omega^R(\mathcal{\textbf{x}}^0,\dots,\mathcal{\textbf{x}}^p,\textbf{t}) \wedge
  \Theta_1(\textbf{t},\textbf{a}) \wedge
  \Theta_2^\Phi (\mathcal{\textbf{x}}^0,\dots,\mathcal{\textbf{x}}^p,\textbf{t})
\},
\end{aligned}
$$
where
\begin{equation}
{\label{eqn:semi-algebraic-join}}
\begin{aligned}
\Omega^R \ :=\ & \quad \bigwedge_{i=0}^p (|\X^i|^2\leq R^2) \wedge |\T|^2\leq
1, \\
\Theta_1 \ :=\ & \quad \sum_{i=0}^p T_i=1 \wedge
\sum_{0 \leq i < j \leq p} A_{i j}^2=0, \\
\Theta_2^\Phi\ :=\ & \quad \bigwedge_{i=0}^p (T_i=0 \vee  \Phi(\X^i)), \\
\end{aligned}\end{equation}
We denote the formula $\Omega^R\wedge \Theta_1\wedge \Theta_2^\Phi$
by $\mathcal{J}^p(\Phi)$.
\end{definition}

It is checked easily from Definition \ref{def:semi-algebraic-join} that
$$
\displaylines{
\mathcal{J}^p(X) \subset \left(\overline{B_{k}(0,R)}\right)^{p+1} \times \Delta'_{[0,p]} \times
\{\textbf{0}\},
}
$$
and that the deformation retraction
$\rho_p: \Delta'_{[0,p]} \rightarrow \Delta_{[0,p]}$
extends to a deformation retraction,
$\tilde{\rho}_p: \mathcal{J}^p(X) \rightarrow \tilde{\mathcal{J}}^p(X)$, where
$\tilde{\mathcal{J}}^p(X)$ is defined by
$$
\displaylines{
\tilde{\mathcal{J}}^p(X)=
\{
(\mathcal{\textbf{x}}^0,\dots,\mathcal{\textbf{x}}^p,
  \textbf{t},\textbf{a})\in \left(\overline{B_{k}(0,R)}\right)^{p+1} \times \Delta_{[0,p]} \times
\{\textbf{0}\} \;\mid\;
   \Theta_2^\Phi (\mathcal{\textbf{x}}^0,\dots,\mathcal{\textbf{x}}^p,\textbf{t})
\}.
}
$$
Finally, it is a consequence of the Vietoris-Beagle theorem
(see \cite[Theorem 2]{BWW06})
that
$\tilde{\mathcal{J}}^p(X)$ and $J^p(X)$ are homotopy equivalent.
We thus have, using notation introduced
above, that
\begin{proposition}
$\mathcal{J}^p(X)$ is homotopy equivalent to $J^p(X)$.
\end{proposition}

\begin{remark}
The necessity of defining $\mathcal{J}^p(X)$ instead of just
$\tilde{\mathcal{J}}^p(X)$ has to do with removing the inequalities
defining the standard simplex from the defining formula
$\mathcal{J}^p(\Phi)$, and this will
simplify certain arguments later in the paper.
\end{remark}

We now generalize the above constructions and define joins over maps (the
topological and semi-algebraic joins defined above are special cases when the
map is a constant map to a point).

\begin{notationdefinition}
\label{notdef:fiberproduct}
\emph{
Let $f:A \rightarrow B$ be a map
between topological spaces $A$ and $B$.
For each $p \geq 0$,
we denote by $W_f^p(A)$ the
\emph{$(p+1)$-fold fiber product} of $A$ over $f$.
In other words
\[
W_f^p(A) = \{(x_0,\ldots,x_p) \in A^{p+1}
\mid
f(x_0) = \cdots = f(x_p)
\}.
\]}
\end{notationdefinition}

\begin{definition}[Topological join over a map]
\label{def:joinoveramap1}
Let $f:X \rightarrow Y$ be a map
between topological spaces $X$ and $Y$.
For
$p \geq 0$,
the
\emph{$(p+1)$-fold join}
$J^p_f(X)$  of $X$ over $f$
is
\begin{equation}
\label{eqn:definitionofjoin1}
J^p_f(X) \defeq W_f^p(X) \times \Delta^p/\sim,
\end{equation}
where
\[
(x_0,\ldots,x_p,t_0,\ldots,t_p) \sim (x_0',\ldots,x_p',t_0,\ldots,t_p)
\]
if for each $i$ with $t_i \neq 0$, $x_i = x_i'$.
\end{definition}

In the special situation
when $f$ is a semi-algebraic continuous map,
the space $J^p_f(X)$ defined above is (as before) not immediately
a semi-algebraic set, because of taking quotients.
Our next goal is to obtain a semi-algebraic set,
$\mathcal{J}^p_{f}(X)$ which is homotopy equivalent
to $J^p_f(X)$ similar to the case of the ordinary join.

\newcommand{\thetafour}{\bigwedge_{0\leq i<j \leq p } (T_i= 0 \vee T_j = 0 \vee
|f(\X^i)-f(\X^j)|^2=A_{ij})}

\begin{definition}[The semi-algebraic fibered join \cite{BZ09}%
    \label{defnJP}]
For a semi-algebraic subset $ X\subset \mathbb{R}^k$
contained in $B_k(0,R)$,
defined by a $\mathcal{P}$-formula
$\Phi$ and $f:X\to Y$ a semi-algebraic map, we define
$$
  \begin{aligned}\mathcal{J}^p_{f}(X)=&\{
(\textbf{x}^0,\dots,\textbf{x}^p,
  \textbf{t},\textbf{a})\in  \mathbb{R}^{(p+1)(k+1)+\binom{p+1}{2}}|\\ &\quad
  \Omega^R(\mathcal{\textbf{x}}^0,\dots,\mathcal{\textbf{x}}^p,\textbf{t}) \wedge
  \Theta_1(\textbf{t},\textbf{a}) \wedge
  \Theta_2^\Phi (\mathcal{\textbf{x}}^0,\dots,\mathcal{\textbf{x}}^p,\textbf{t})\wedge
  \Theta_3^f
  (\mathcal{\textbf{x}}^0,\dots,\mathcal{\textbf{x}}^p,\textbf{t},\textbf{a})
\},
\end{aligned}$$ where
$\Omega^R, \Theta_1,
\Theta_2^\Phi$ have been defined previously, and
\begin{equation}{\label{eqn:fibjoin}}\begin{aligned}
\Theta_3^f \ := \  & \thetafour.
\end{aligned}\end{equation}
We denote the formula $\Omega^R\wedge \Theta_1\wedge \Theta_2^\Phi \wedge
\Theta_3^f$ by $\mathcal{J}^p_f(\Phi)$.
\end{definition}

Observe that there exists a natural map,
$J^p(f): \mathcal{J}^p_f(X) \rightarrow Y$, which maps a point
$(\textbf{x}^0,\dots,\textbf{x}^p, \textbf{t},\textbf{0}) \in
\mathcal{J}^p_f(X)$ to $f(\textbf{x}^i)$ (where $i$ is such that $t_i \neq 0$).
It is easy to see that for each $\textbf{y} \in Y$,
$J^p(f)^{-1}(\textbf{y}) = \mathcal{J}^p(f^{-1}(\textbf{y}))$.

The following proposition
follows from the above observation and the generalized
Vietoris-Begle theorem
(see \cite[Theorem 2]{BWW06})
and is important in the proof of Proposition \ref{prop:main};
it relates up to $p$-equivalence
the semi-algebraic set $\mathcal{J}^p_f(X)$ to the
image of a closed, continuous semi-algebraic surjection $f:X\to Y$.
Its proof
is similar to the proof of
Theorem  2.12 proved in \cite{BZ09} and is omitted.

\begin{proposition}
{\label{prop:5}}\cite{BZ09}
Let $f:X \to Y$ a closed, continuous
semi-algebraic
surjection
with
$X \subset B_k(0,R)$
a closed
semi-algebraic set.  Then,
for  every $p\geq 0$,
the map $J^p(f):  \mathcal{J}^p_f(X) \to Y$ is a $p$-equivalence.
\end{proposition}

We now define a thickened version of the semi-algebraic set
$\mathcal{J}^p_f(X)$ defined above and prove that it is homotopy equivalent to
$\mathcal{J}^p_f(X)$.
The variables $A_{ij}, 0 \leq i < j \leq p$, play an important role in
the thickening process.

\begin{definition}[The thickened semi-algebraic fibered join]
\label{defncalJ}
For $X\subset \mathbb{R}^k$ a semi-algebraic
  set contained in $B_k(0,R)$
  defined by a $\mathcal{P}$-formula
$\Phi$,
  $p\geq 1$, and $\ep>0$
  define
  $$
  \begin{aligned}\mathcal{J}^p_{f,\ep}(X)=&\{(\mathcal{\textbf{x}}^0,\dots,\mathcal{\textbf{x}}^p,
  \textbf{t},\textbf{a})\in  \mathbb{R}^{(p+1)(k+1)+\binom{p+1}{2}}|\\ &\quad
\Omega^R(\mathcal{\textbf{x}}^0,\dots,\mathcal{\textbf{x}}^p,\textbf{t}) \wedge
  \Theta_1^\ep(\textbf{t},\textbf{a}) \wedge
  \Theta_2^\Phi (\mathcal{\textbf{x}}^0,\dots,\mathcal{\textbf{x}}^p,\textbf{t})\wedge
  \Theta_3(\mathcal{\textbf{x}}^0,\dots,\mathcal{\textbf{x}}^p,\textbf{t},\textbf{a})
  \},
  \end{aligned}$$ where
\begin{equation}{\label{eqn:fibjoin2}}\begin{aligned}
\Omega^R\ :=\ &\quad \bigwedge_{i=0}^p (|\X^i|^2\leq R^2) \wedge |\T|^2\leq
1, \\
\Theta_1^\ep\ :=\ &\quad \sum_{i=0}^p T_i=1 \wedge \sum_{1\leq i<j\leq p} A_{i j}^2\leq \ep, \\
\Theta_2^\Phi\ :=\ & \quad \bigwedge_{i=0}^p (T_i=0 \vee  \Phi(\X^i)), \\
\Theta_3^f \ := \  & \thetafour.
\end{aligned}\end{equation}
\end{definition}

Note that if $X$ is closed
(and bounded), then
$\mathcal{J}^p_{f,\ep}(X)$ is again closed (and bounded).

The relation between $\mathcal{J}^p_{f}(X)$ and $\mathcal{J}^{p}_{f,\ep}(X)$ is
described in the following proposition.

\begin{proposition}
\label{prop:7}
For $p\in \mathbb{N}$,  $f:X\to Y$
  semi-algebraic  there exists $\ep_0>0$ such that $\mathcal{J}^p_{f}(X)$ is
homotopy equivalent to $\mathcal{J}^{p}_{f,\ep}(X)$
  for all $0<\ep\leq \ep_0$.
\end{proposition}

Proposition \ref{prop:7} follows from the following two lemmas.

\begin{lemma}{\label{lem:prop7b}} For $p\in \mathbb{N}$, $f:X\to Y$
  semi-algebraic we have
  $$\mathcal{J}^p_f(X)=\bigcap_{t>0} \mathcal{J}^p_{f,t}(X).$$
\end{lemma}

\proof Obvious from Definitions \ref{defnJP} and \ref{defncalJ}. \zz

\begin{lemma}
\label{lem:prop7a}  Let
$\TT \subset \mathbb{R}^k\times \mathbb{R}_+$
  such that each $\TT_t$ is closed and $\TT_t\subseteq B_k(0,R)$
  for $t>0$.   Suppose further that for all $0<t\leq t'$ we have $\TT_t
  \subseteq \TT_{t'} $.
Then, $$\bigcap_{t>0}\TT_t =
  \pi_{[1,k]}\left(\overline{\TT}\cap \pi_{k+1}^{-1}(0)\right).$$
  Furthermore,
  there exists  $\ep_0>0$ such that for all $\ep$ satisfying
  $0<\ep\leq \ep_0$  we have that
$\TT_\ep$
  is
semi-algebraically homotopy
  equivalent to
$\TT_\limit$
(cf. Notation \ref{not:limit}).
\end{lemma}

\proof 
The first part of the proposition is straightforward.  The second part follows easily from Lemma 16.16 in \cite{BPRbook2}.

\zz

\proof[Proof of Proposition \ref{prop:7}]
The set $\TT=\{(\mathcal{\textbf{x}},t)\in
\re^{k+1}| \ t>0 \wedge \mathcal{\textbf{x}}\in \mathcal{J}^p_{f,t}(X)\}$ satisfies the conditions
of Lemma  \ref{lem:prop7a}.  The proposition now follows from Lemma
\ref{lem:prop7a}  and Lemma \ref{lem:prop7b}.
\zz

\begin{proposition}
\label{prop:8}
For $p\in\mathbb{N}$, $f:X\to Y$ semi-algebraic, and
$0<t\leq t'$,
  $$\mathcal{J}^p_{f, t}(X)\subseteq \mathcal{J}^p_{f,t'}(X).$$
  Moreover,  there exists $\ep_0>0$ such that for
$0<\ep\leq \ep'<\ep_0$
the above inclusion induces a semi-algebraic homotopy equivalence.
\end{proposition}

The first part of Proposition \ref{prop:8} is obvious
from the definition of $\mathcal{J}^p_{f, \ep}(X)$.
The second part follows from Lemma \ref{lem:2} below.

The following lemma is probably well known and easy.
However, since we were unable to locate
an exact statement to this effect in the literature,
we include a proof.

\begin{lemma}
\label{lem:2}
Let
$\TT\subset \mathbb{R}^k \times \mathbb{R}_+$
be a
semi-algebraic set, and
suppose that
$\TT_t \subset \TT_{t'}$
for all $0<t<t'$. Then, there
exists $\ep_0$ such that for each $0<\ep<\ep'\leq \ep_0$ the inclusion
map
$\TT_\ep \overset{i_{\ep'}}{\hookrightarrow} \TT_{\ep'}$
induces a
semi-algebraic homotopy equivalence.
\end{lemma}

\proof
We prove that
there exists $\phi_{\ep'}:\TT_{\ep'}\to
\TT_{\ep}$ such that $$\begin{aligned}\phi_{\ep'}\circ i_{\ep'}&:\TT_\ep
\to \TT_{\ep}\halfspace, \quad \phi_{\ep'}\circ i_{\ep'} \simeq \Id_{\TT_{\ep}}, \\
i_{\ep'} \circ \phi_{\ep'}&: \TT_{\ep'} \to \TT_{\ep'}, \quad i_{\ep'}\circ
\phi_{\ep'}\simeq \Id_{\TT_{\ep'}}.
\end{aligned}
$$

We first define $i_t:\TT_\ep \hookrightarrow \TT_t$ and $\widehat{i}_t:\TT_t
\hookrightarrow \TT_{\ep'}$, and note that trivially $i_\ep=
\Id_{\TT_\ep}$, $\widehat{i}_{\ep'}= \Id_{\TT_{\ep'}}$, and
$i_{\ep'}=\widehat{i}_\ep$.  Now, by Hardt triviality there exists
$\ep_0>0$, such that there is a definably trivial homeomorphism $h$ which
commutes with the projection
$\pi_{k+1}$,
i.e., the following diagram commutes.
$$
\xymatrix{\TT_{\ep_0} \times (0,\ep_0] \ar[r]^{h
\phantom{stuffhere}} \ar[d]^{\pi_{k+1}} & \TT \cap \{(\mathcal{\textbf{x}},t)| \ 0<t\leq
\ep_0\} \ar[dl]_{\pi_{k+1}} \\ (0,\ep_0] }
$$

Define
$F(\mathcal{\textbf{x}},t,s)=h(\pi_{[1,k]}\circ h^{-1}(\mathcal{\textbf{x}},t),s)$.  Note that $F(\mathcal{\textbf{x}},t,t)=h
(\pi_{[1,k]} \circ h^{-1}(\mathcal{\textbf{x}},t),t)=h(h^{-1}(\mathcal{\textbf{x}},t))=(\mathcal{\textbf{x}},t)$.
We define
$$\begin{aligned}\phi_t:\TT_t&\to \TT_\ep, \\ \phi_t(\mathcal{\textbf{x}})&=\pi_{[1,k]}\circ
F(\mathcal{\textbf{x}},t,\ep), \\ \widehat{\phi}_t: \TT_{\ep'}& \to \TT_t, \\
\widehat{\phi}_t(\mathcal{\textbf{x}})&=\pi_{[1,k]}\circ F(\mathcal{\textbf{x}},\ep',t)\end{aligned}$$ and
note that
$\phi_{\ep'}=\widehat{\phi}_\ep$.

Finally,
define $$\begin{array}{ll}
H_1(\cdot,t)&= \phi_t\circ i_t :\TT_\ep \to \TT_\ep, \\
H_1(\cdot,\ep)&=\phi_\ep\circ i_\ep = \Id_{\TT_\ep}, \\
H_1(\cdot,\ep')&=\phi_{\ep'}\circ i_{\ep'}, \\
\ & \ \\
H_2(\cdot,t)&= \widehat{i}_t \circ \widehat{\phi}_t: \TT_{\ep'} \to  \TT_{
  \ep'}, \\
H_2(\cdot,\ep)&=\widehat{i}_\ep \circ \widehat{\phi}_\ep = i_{\ep'}
\circ \phi_{ \ep'}, \\
H_2(\cdot,\ep')&=\widehat{i}_{\ep'} \circ \widehat{\phi}_{\ep'} =
\Id_{\TT_{ \ep'}}. \end{array}$$

The semi-algebraic continuous maps $H_1$ and $H_2$ defined above
give a semi-algebraic homotopy between the maps
$\phi_{\ep'} \circ i_{\ep'} \simeq  \Id_{\TT_{\ep}}$ and
$i_{\ep'} \circ \phi_{ \ep'} \simeq \Id_{\TT_{\ep'}}$
proving the required semi-algebraic homotopy equivalence.
\zz

\newcommand{\Upsilonthree}
{\bigwedge_{0\leq i<j \leq p }
     (T_i= 0 \vee T_j = 0
\vee |\X^i-\X^j|^2= A_{i j})}

As mentioned before, we would like to
replace $\mathcal{J}^p_{f,\eps}(X)$
by another semi-algebraic set,
which we denote by
$\mathcal{D}^p_\eps(X)$,
which is  homotopy equivalent to
$\mathcal{J}^p_{f,\eps}(X)$,
under certain assumptions on $f$ and $\eps$,
whose definition no longer involves the
map $f$. This is what we do next.

\begin{definition}
[The thickened diagonal]
\label{defncalD}
For a semi-algebraic set $X\subset \mathbb{R}^k$ contained in $B_k(0,R)$
defined by a $\mathcal{P}$-formula
$\Phi$, 
$p\geq 1$, and $\ep>0$,
define
$$
  \begin{aligned}\mathcal{D}^p_{\ep}(X)=&\{(\mathcal{\textbf{x}}^0,\dots,\mathcal{\textbf{x}}^p,
  \textbf{t},\textbf{a})\in  \mathbb{R}^{(p+1)(k+1)+\binom{p+1}{2}}|\\ &\quad
\Omega^R(\mathcal{\textbf{x}}^0,\dots,\mathcal{\textbf{x}}^p,\textbf{t}) \wedge
  \Theta_1(\textbf{t},\textbf{a}) \wedge
  \Theta_2^\Phi (\mathcal{\textbf{x}}^0,\dots,\mathcal{\textbf{x}}^p,\textbf{t})\wedge
  \Upsilon(\mathcal{\textbf{x}}^0,\dots,\mathcal{\textbf{x}}^p,\textbf{t},\textbf{a})
  \},
    \end{aligned}$$
    where $\Omega^R,\Theta_1^\ep,\Theta_2^\Phi$ are defined as in Equation \ref{eqn:fibjoin2}, and
\begin{equation*}{
\label{eqn:calD}
}\begin{aligned}
\Upsilon \ := \  & \Upsilonthree.
\end{aligned}\end{equation*}
\end{definition}

Notice that the formula defining the thickened diagonal,
$\mathcal{D}^p_{\ep}(X)$
in Definition
\ref{defncalD},
is identical
to that defining the thickened semi-algebraic fibered join,
$\mathcal{J}^p_{f,\ep}(X)$
in Definition \ref{defncalJ},
except that $\Theta_3^f$ is replaced by $\Upsilon$, and
$\Upsilon$ does not depend on the map $f$ or on the set $X$.

\begin{proposition}
{\label{prop:calDp} \label{prop:9}}
Let $X \subset \mathbb{R}^k$ be a semi-algebraic set defined by a
quantifier free
formula
$\Phi$ having
(division-free) additive format bounded by $(a,k)$
and dense format bounded by $(s,d,k)$.
 Then,
$\mathcal{D}^p_\eps(X)$
is a semi-algebraic subset set of
$\mathbb{R}^N$, defined by a formula
with
(division-free)
 additive format bounded by $(M,N)$ and dense format
bounded by   $(M',d+1,N)$,  where
$M=\constM$, $M'=\constMprime$, and
  $N=(p+1)(k+1)+\binom{p+1}{2}$.
\end{proposition}

\proof
It is a straightforward computation to bound the division-free additive
format and give the dense format of the formulas
$\Omega^R,\Theta_1^\eps, \Upsilon$ as well as the (division-free) additive format and dense format of the formula
 $\Theta_2^\Phi$.
More precisely, let
$$\begin{array}{ll}
M_{\Omega^R} = (p+1)k+(p+1), & M'_{\Omega^R} = (p+1)+1, \\
M_{\Theta_1^\ep}=(p+1)+\binom{p+1}{2}, & M'_{\Theta_1^\ep}=2, \\
M_{\Theta_2^\Phi}=(p+1)a, & M'_{\Theta_2^\Phi}=(p+1)(s+1), \\
M_{\Upsilon}=2k\binom{p+1}{2}, &M'_{\Upsilon}=3\binom{p+1}{2}.
\end{array}$$
It is clear from
Definition \ref{defncalD}
that the
division-free additive format (resp. dense format) of $\Omega^R$ is
bounded by $(M_{\Omega^R},N)$, $N=(p+1)(k+1)+\binom{p+1}{2}$
(resp. $(M'_{\Omega^R},2,N)$).  Similarly, the division-free
additive format (resp.
 dense format) of $\Theta_1^\ep,
 \Upsilon$
is bounded by $(M_{\Theta_1^\ep},N),
(M_{\Upsilon},N)$
 (resp. $(M'_{\Theta_1^\ep},2,N)$,
 $(M'_{\Upsilon},2,N)$).
 Finally, the (division-free) additive format of
 $\Theta_2^\Phi$
 is bounded by $(M_{\Theta_2^\Phi},N)$
 and dense format is $(M'_{\Theta_2^\Phi},d+1,N)$.
The (division-free) additive
 format (resp. dense format) of the formula defining
$\mathcal{D}^p_\ep(X)$
 is thus bounded by
$$\left(M_{\Omega^R}+M_{\Theta_1^\ep}+M_{\Theta_2^\Phi}
+M_{\Upsilon},N\right) \qquad (\text{resp.} (M'_{\Omega^R}+M'_{\Theta_1^\ep}+M'_{\Theta_2^\Phi}
+M'_{\Upsilon},d+1,N)).$$
\zz

We now relate the thickened semi-algebraic fibered-join and
the thickened diagonal
using a sandwiching argument similar in spirit to that used in
\cite{Hausdorff}.

\begin{subsubsection}
{Limits of one-parameter families}
In this section, we fix a bounded semi-algebraic set
$\TT\subset \mathbb{R}^k\times \mathbb{R}_+$
such that
$\TT_t$ is closed and  $\TT_t
\subseteq B_{k}(0,R)$ for some
$R\in \mathbb{R}_+$ and all $t>0$.
Let $\TT_\limit$ be as in Notation \ref{not:limit}.

We need the following proposition proved in \cite{Hausdorff}.

\begin{proposition}[\cite{Hausdorff} Proposition 8]
\label{prop:12}
There
exists $\lambda_0>0$ such that for every $\lambda \in (0,\lambda_0]$
there exists a continuous semi-algebraic surjection $f_\lambda: \TT_\lambda
\to
\TT_\limit
$ such that the family of maps
 $\{f_\lambda\}_{0<\lambda\leq \lambda_0}$
satisfies
\begin{enumerate}
\item
$$
\lim_{\lambda\to 0} \max_{\emph{\textbf{x}}\in
  \TT_\lambda}  |\emph{\textbf{x}}-f_\lambda(\emph{\textbf{x}})|=0,
$$
and
\item
for each $\lambda,\lambda'\in (0,\lambda_0)$, $f_\lambda=f_{\lambda'}\circ g$
for some semi-algebraic homeomorphism
$g :\TT_\lambda \to \TT_{\lambda'}$.
\end{enumerate}
\end{proposition}

\begin{proposition}
\label{prop:13}
There exist
  $\lambda_1$  satisfying $0<\lambda_1\leq \lambda_0$ and semi-algebraic
  functions  $\delta_0,\delta_1: (0,\lambda_1) \rightarrow \mathbb{R}$,
  such that
\begin{enumerate}
\item $0 < \delta_0(\lambda) < \delta_1(\lambda)$, for $\lambda \in (0,\lambda_1)$,
\item  $\lim_{\lambda \to 0}
 \delta_0 (\lambda)=0$,
  $\lim_{\lambda \to 0} \delta_1(\lambda)\neq
  0$,
\item
for  each $\lambda\in (0,\lambda_1)$,
and $\delta, \delta'$ satisfying
$0<\delta_0(\lambda)<\delta<\delta'<  \delta_1(\lambda)$,
the inclusion  $\mathcal{D}^p_{\delta'}(\TT_\lambda)
  \hookrightarrow \mathcal{D}^p_{\delta}(\TT_\lambda)$
induces a semi-algebraic homotopy equivalence.
\end{enumerate}
\end{proposition}

Proposition \ref{prop:13} is adapted from Proposition 20 in
\cite{Hausdorff}
and the proof is identical after replacing
$D^p_\lambda(\delta)$ (defined in \cite{Hausdorff}) with
the semi-algebraic set $\mathcal{D}^p_{\delta}(\TT_\lambda)$ defined above
(Definition \ref{defncalD}).

Let $f_\lambda$, $\lambda\in (0,\lambda_0]$, satisfy the conclusion of Proposition \ref{prop:12}.
As in \cite{Hausdorff}, define for $p\in \mathbb{N}$
\begin{equation}
\eta_p(\lambda)=p(p+1)\left(4R \max_{\mathcal{\textbf{x}} \in T_\lambda} |
\mathcal{\textbf{x}}-f_\lambda(\mathcal{\textbf{x}})| + 2\left(\max_{\mathcal{\textbf{x}} \in T_\lambda } |
\mathcal{\textbf{x}}-f_\lambda(\mathcal{\textbf{x}})|\right)^2\right).
\end{equation}

Note that, for every $\lambda\in (0,\lambda_0]$ and every $q\leq p$,
we have $\eta_q(\lambda) \leq \eta_p(\lambda)$.  Additionally, for each $p\in \mathbb{N}$,
$\lim_{\lambda \to 0 } \eta_p(\lambda)=0$ by Proposition \ref{prop:12} A.

Define for $\overline{\textbf{x}}=(\mathcal{\textbf{x}}^0,\dots,\mathcal{\textbf{x}}^p)\in \mathbb{R}^{(p+1)k}$ the sum
$\rho_p(\overline{\textbf{x}})$ as $$\rho_p(\mathcal{\textbf{x}}^0,\dots,\mathcal{\textbf{x}}^p)=\sum_{1\leq i < j
\leq p} |\mathcal{\textbf{x}}^i - \mathcal{\textbf{x}}^j|^2.$$
A special case of this sum corresponding to all $t_i\neq 0$
appears in the formula $\Upsilon^\ep_1$ of Definition \ref{defncalD}
after making the replacement $a_{ij}=|\mathcal{\textbf{x}}^i-\mathcal{\textbf{x}}^j|$.
The next lemma is taken from \cite{Hausdorff} to which we refer
the reader for the proof.

\begin{lemma}[\cite{Hausdorff} Lemma 21]\label{lem:7}  Given $\eta_p(\lambda)$
  and  $f_\lambda:\TT_\lambda \to \TT_\limit$ as above, we have
$$|\halfhalfspace \sum_{i<j} |f_\lambda(\emph{\textbf{x}}^i)-f_\lambda(\emph{\textbf{x}}^j)|^2-
  \sum_{i<j} |\emph{\textbf{x}}^i-\emph{\textbf{x}}^j|^2\halfhalfspace| \halfspace \leq \halfspace
  \eta_p(\lambda),
$$
and in particular
$$\rho_p(\emph{\textbf{x}}^0,\dots,\emph{\textbf{x}}^p)\halfhalfspace \leq \halfhalfspace
\rho_p(f_\lambda(\emph{\textbf{x}}^0),\dots,f_\lambda(\emph{\textbf{x}}^p))+\eta_p(\lambda)
\halfhalfspace \leq \halfhalfspace \rho_p(\emph{\textbf{x}}^0,\dots,\emph{\textbf{x}}^p) + 2\eta_p(\lambda).
$$
\end{lemma}

The next proposition follows immediately from Lemma \ref{lem:7},
Definition \ref{defncalJ}, and Definition \ref{defncalD}.

\begin{proposition}
\label{propINC}
For every $\lambda\in (0,\lambda_0)$ and $\ep>0$, we have
  $$ \mathcal{J}^p_{f_\lambda,\ep}(\TT_\lambda) \subseteq
  \mathcal{D}^p_{\ep+\eta_p(\lambda)} (\TT_\lambda) \subseteq
  \mathcal{J}^p_{f_\lambda,\ep+2 \eta_p (\lambda)}(\TT_\lambda).$$
\end{proposition}

Let $\ep_1,\ep_2\in \mathbb{R}_+$ satisfy the conclusions of Proposition \ref{prop:7},
 Proposition \ref{prop:8}, respectively.  Set $\ep_0=\min\{\ep_1,\ep_2\}$.

\begin{proposition}
\label{prop:15}
For any $p\in
  \mathbb{N}$, there exist
$\lambda,\ep,\delta\in \mathbb{R}_+$ such that
$\ep\in (0,\ep_0)$, $\lambda\in (0,\lambda_0)$, and
 $$\mathcal{D}^p_\delta(\TT_\lambda) \simeq
  \mathcal{J}^p_{f_\lambda,\ep}(\TT_\lambda).$$
\end{proposition}

\proof We first describe how to choose
$\ep,\ep'\in (0,\ep_0)$, $\lambda\in (0,\lambda_0)$ and $\delta,\delta' \in
(\delta_0(\lambda),\delta_1(\lambda))$
(cf. Proposition \ref{prop:13}) so that $$
\mathcal{D}^p_{\delta'}(\TT_\lambda) \subseteq \mathcal{J}^p_{f_\lambda,\ep}(\TT_\lambda)
\overset{\ast}{\subseteq} \mathcal{D}^p_{\delta}(\TT_\lambda)\subseteq
\mathcal{J}^p_{f_\lambda,\ep'}(\TT_\lambda),$$ and secondly we show
that,
with these choices,
the inclusion $(\ast)$ induces a homotopy equivalence.

Since the limit of $\delta_1(\lambda)-\delta_0(\lambda)$ is not zero
for $0<\lambda<
\lambda_1\leq \lambda_0
$ and $\lambda$ tending to zero, while the
limits of $\eta_p(\lambda)$ and $\delta_0(\lambda)$ are zero
(by  Proposition \ref{prop:13}, Proposition \ref{prop:12} A), we can
choose $0<\lambda<\lambda_0$ which simultaneously satisfies
$$2\eta_p(\lambda)<\fraction{\delta_1(\lambda)-\delta_0(\lambda)}{2} \qquad
\text{and} \qquad \delta_0(\lambda)+4\eta_p(\lambda)<\ep_0.$$

Set
$\delta'=\delta_0+\eta_p(\lambda)$,
$\ep=\delta_0+2\eta_p(\lambda)$,
$\delta=\delta_0+3\eta_p(\lambda)$, and $\ep'=\delta_0+4\eta_p(\lambda)$.
From Proposition \ref{propINC} we have the following inclusions,
$$
\mathcal{D}^p_{\delta'}(\TT_\lambda) \overset{i}{\hookrightarrow}
\mathcal{J}^p_{f_\lambda,\ep}(\TT_\lambda) \overset{j}{\hookrightarrow}
\mathcal{D}^p_{\delta}(\TT_\lambda)\overset{k}{\hookrightarrow}
\mathcal{J}^p_{f_\lambda,\ep'}(\TT_\lambda).$$

Furthermore,
it is easy to see that $\delta,
\delta'\in (\delta_0(\lambda),\delta_1(\lambda))$ and that
$\ep,\ep'\in (0,\ep_0)$, and so we have that both $j\circ i$ and
$k\circ j$ induce semi-algebraic homotopy equivalences (Proposition \ref{prop:13},
Proposition \ref{prop:8} resp.).

For each $\textbf{z} \in \mathcal{D}^p_{\delta'}(\TT_\lambda)$ we have the following
diagram between the homotopy groups.

$$
\xymatrix{ \boldsymbol{\pi}_\ast(\mathcal{D}^p_{\delta'}(\TT_\lambda),\textbf{z})
\ar[rr]^{(j\circ i )_\ast}_{\cong} \ar[rd]^{i_\ast} & &
 \boldsymbol{\pi}_\ast(\mathcal{D}^p_{\delta}(\TT_\lambda), \textbf{z}) \ar[rd]^{k_\ast} & \\ &
 \boldsymbol{\pi}_\ast(\mathcal{J}^p_{f_\lambda,\ep}(\TT_\lambda),\textbf{z}) \ar[rr]^{(k\circ j)_\ast}_{\cong}
 \ar[ru]^{j_\ast} & & \boldsymbol{\pi}_\ast(\mathcal{J}^p_{f_\lambda,\ep'}(\TT_\lambda),
\textbf{z})}
$$
where we have identified $\textbf{z}$ with its images under the various
inclusion maps.

Since $(j\circ i)_\ast=j_\ast \circ i_\ast$, the surjectivity of
$(j\circ i)_\ast$ implies that $j_\ast$ is surjective, and similarly
$(k\circ j)_\ast$ is injective
ensures that $j_\ast$ is injective.
Hence, $j_\ast$ is an isomorphism as required.

This implies that
the inclusion map
$\mathcal{J}^p_{f_\lambda,\ep}(\TT_\lambda)\overset{j}{\hookrightarrow} \mathcal{D}^p_\delta(\TT_\lambda)$
 is a
weak homotopy equivalence
(see \cite[page 181]{Whitehead}). Since
both spaces have the structure of a finite
CW-complex, every weak equivalence is in fact a homotopy equivalence
(\cite[Theorem 3.5, p. 220]{Whitehead}).
\zz

\end{subsubsection}

We now prove Proposition \ref{prop:main}.
\proof [Proof of Proposition \ref{prop:main}]
Let
$\TT\subset \mathbb{R}^k \times \mathbb{R}_+$
such that $\TT_\lambda$ is closed and $\TT_\lambda\subset B_k(0,R)$ for some $R\in \mathbb{R}$
and all
$\lambda\in \mathbb{R}_+$.
Applying Proposition \ref{prop:15},
we have that there exist $\lambda\in
(0,\lambda_0)$ and $\ep\in (0,\ep_0)$
such that
 the sets $\mathcal{D}^p_\delta(\TT_\lambda)\simeq
\mathcal{J}^p_{f_\lambda,\ep}(\TT_\lambda)$ are
semi-algebraically
homotopy equivalent.  Also, by
Proposition \ref{prop:7} the sets
$\mathcal{J}^p_{f_\lambda,\ep}(\TT_\lambda)\simeq \mathcal{J}^p_{f_\lambda}(\TT_\lambda)$ are
semi-algebraically homotopy equivalent.
By Proposition \ref{prop:5} and Proposition \ref{prop:12} the map
$J(f_\lambda):\mathcal{J}^p_{f_\lambda}(\TT_\lambda)\twoheadrightarrow \TT_0$ induces a
$p$-equivalence.

Thus we have the following sequence of homotopy equivalences and
$p$-equivalence.
\begin{equation}{\label{eqn:mainresult}}
\mathcal{D}^p_{\delta}(\TT_\lambda)\simeq
\mathcal{J}^p_{f_\lambda,\ep}(\TT_\lambda)
\simeq\mathcal{J}^p_{f_\lambda}(\TT_\lambda)
\halfspace {{\widetilde{\to\halfbackspace \succ}_p}}\halfspace \TT_{\limit}\end{equation}

The first homotopy equivalence follows from Proposition \ref{prop:15},
the second from Proposition \ref{prop:7}, and the last $p$-equivalence
is a consequence of Propositions \ref{prop:5} and \ref{prop:12}.
The bound on the format
of the
formula defining $\mathcal{D}^p:=\mathcal{D}^p_\delta(\TT_\lambda)$ follows from
Proposition \ref{prop:9}.  This finishes the proof.
\zz

\proof[Proof of Theorem \ref{thm:main_weak}]
The theorem follows directly from Proposition \ref{prop:main},
Theorem \ref{thm:additive} and
Proposition \ref{prop:top_basic} after choosing $p=k+1$.
\zz

\begin{section}
{Proofs of Theorem \ref{thm:main1} and Theorem \ref{thm:main}
\label{sec:main1}}

\subsection{Algebraic preliminaries}
We start
with a lemma
that provides a slightly different characterization of
additive complexity from that given in Definition
\ref{def:rational_additive}.
Roughly speaking the lemma states that any given
 additive representation of a given polynomial $P$ can be
modified without changing its length to another  additive
representation of $P$ in which any negative exponents occur only
in the very last step. This simplification will be very useful in
what follows.

\begin{lemma}\cite[page 152]{Dries}
\label{lem:equivalence}
For any $P\in \mathbb{R}[X_1,\ldots,X_k]$ and $a\in \mathbb{N}$ we have
  $P$ has  additive complexity at most $a$ if and only if
there exists a sequence of equations (*)
\begin{itemize}
\item[(i)] $Q_1=u_1X_{1}^{\alpha_{11}} \cdots X_{k}^{\alpha_{1k}} +
  v_1X_{1}^{\beta_{11}} \cdots X_{k}^{\beta_{1k}}$,\\ where $u_1, v_1
  \in \mathbb{R}$, and $\alpha_{11}, \ldots ,\alpha_{1k}, \beta_{11}, \ldots ,
  \beta_{1k} \in \N$;

\item[(ii)] $Q_j=u_jX_{1}^{\alpha_{j1}} \cdots X_{k}^{\alpha_{j k}}
  \prod_{1 \le i \le j-1}Q_{i}^{\gamma_{j i}} + v_jX_{1}^{\beta_{j1}}
  \cdots X_{k}^{\beta_{j k}}\prod_{1 \le i \le
    j-1}Q_{i}^{\delta_{ji}}$,\\ where $1 < j \le a$, $u_j, v_j \in
  \mathbb{R}$, and $\alpha_{j1}, \ldots ,\alpha_{j k}, \beta_{j1}, \ldots ,
  \beta_{j k}, \gamma_{ji}, \delta_{ji} \in \N$ for $1 \le i <j$;

\item[(iii)] $P= cX_{1}^{\zeta_{1}} \cdots X_{k}^{\zeta_{k}}\prod_{1
  \le j \le a}Q_{j}^{\eta_{j}}$,\\ where $c \in \mathbb{R}$, and $\zeta_1,
  \ldots , \zeta_k, \eta_1, \ldots ,\eta_a \in \mathbb{Z}$.
\end{itemize}
\end{lemma}

\begin{remark}
Observe that in Lemma \ref{lem:equivalence}
all exponents other than those in line (iii) are in $\mathbb{N}$ rather
than in $\mathbb{Z}$ (cf. Definition~\ref{def:rational_additive}).
Observe also that if
a polynomial $P$ satisfies the conditions of the lemma, then
it has  additive complexity at most $a$.
\end{remark}

\newcommand{\hatP}{\widehat{P}}
\newcommand{\hatQ}{\widehat{Q}}

\begin{subsection}{The algebraic case}
Before proving Theorem \ref{thm:main1} it is useful to first consider the
algebraic case separately, since the main technical ingredients
used in the proof of Theorem \ref{thm:main1} are more clearly
visible in this case. With this in mind, in this section we consider the
algebraic case and prove the following theorem, deferring the proof in the
general semi-algebraic case till the next section.

\begin{theorem}{
\label{thm:algebraic}}
 The number of
distinct
homotopy types of
  $\Zer(F,\mathbb{R}^k)$ amongst all polynomials $F\in \mathbb{R}[X_1,\dots,X_k]$ having
    additive complexity at most $a$ does not exceed
$$ 2^{O(k(k^2+a))^8} = 
2^{(k+a)^{O(1)}}.$$
\end{theorem}

Before proving Theorem \ref{thm:algebraic} we need a few preliminary results.

\begin{proposition}{\label{prop:sectionfivemain}}
Let $F,P,Q\in \mathbb{R}[\X]$
such that $FQ=P$,
$R \in \mathbb{R}_+$, and define
\begin{equation}{\label{eqn:tee}}
\TT:=\{(\emph{\textbf{x}},t)\in
\mathbb{R}^k \times \mathbb{R}_+ | \; 
P^2(\emph{\textbf{x}}) \leq t(Q^2(\emph{\textbf{x}})-t^N) \wedge 
|\emph{\textbf{x}}|^2\leq R^2
\},
\end{equation}
where $N = 2 \deg(Q)+1$.
Then, using Notation \ref{not:limit}
\[
\TT_{\limit}=\Zer(F,\mathbb{R}^k) \cap \overline{B_k(0,R)}.
\]
\end{proposition}

Before proving Proposition \ref{prop:sectionfivemain}
we first discuss an illustrative example.

\begin{example}
\label{eg:main}
Let
\[
F_1=X(X^2+Y^2-1),
\]
\[F_2=X^2+Y^2-1.
\]
Also, let
\[
P_1=X^2(X^2+Y^2-1),
\]
\[
P_2= X(X^2+Y^2-1),
\]
and
\[
Q_1=Q_2=X.
\]

For $i=1,2$, and $R >0$, let
\[
\TT^i =\{(\mathcal{\textbf{x}},t)\in
\mathbb{R}^k\times \mathbb{R}_+
| \; 
P_i^2(\mathcal{\textbf{x}}) \leq t(Q_i^2(\mathcal{\textbf{x}})-t^N)
\wedge 
|\mathcal{\textbf{x}}|^2
\leq R^2
\}
\]
as in Proposition \ref{prop:sectionfivemain}.

In Figure \ref{fig:example}, we display
from left to right,
$\Zer(F_1,\mathbb{R}^2)$, $\TT^1_\ep$, $\Zer(F_2,\mathbb{R}^2)$ and
and $\TT^2_\ep$, respectively (where $\ep=.005$ and $N=3$).
Notice that, for $i=1,2$ and any fixed $R>0$, the semi-algebraic
set $T^i_\ep$ approaches (in the
sense of Hausdorff
distance)
the set $\ZZ(F_i,\mathbb{R}^2) \cap \overline{B_2(0,R)}$
as $\ep \rightarrow 0$.

\begin{figure}[h!]
\vspace{-150 pt}
\centering
  \subfloat[]{\label{fig:a} \includegraphics[bb=50 0 500 700,
      clip,width=.18\textwidth]{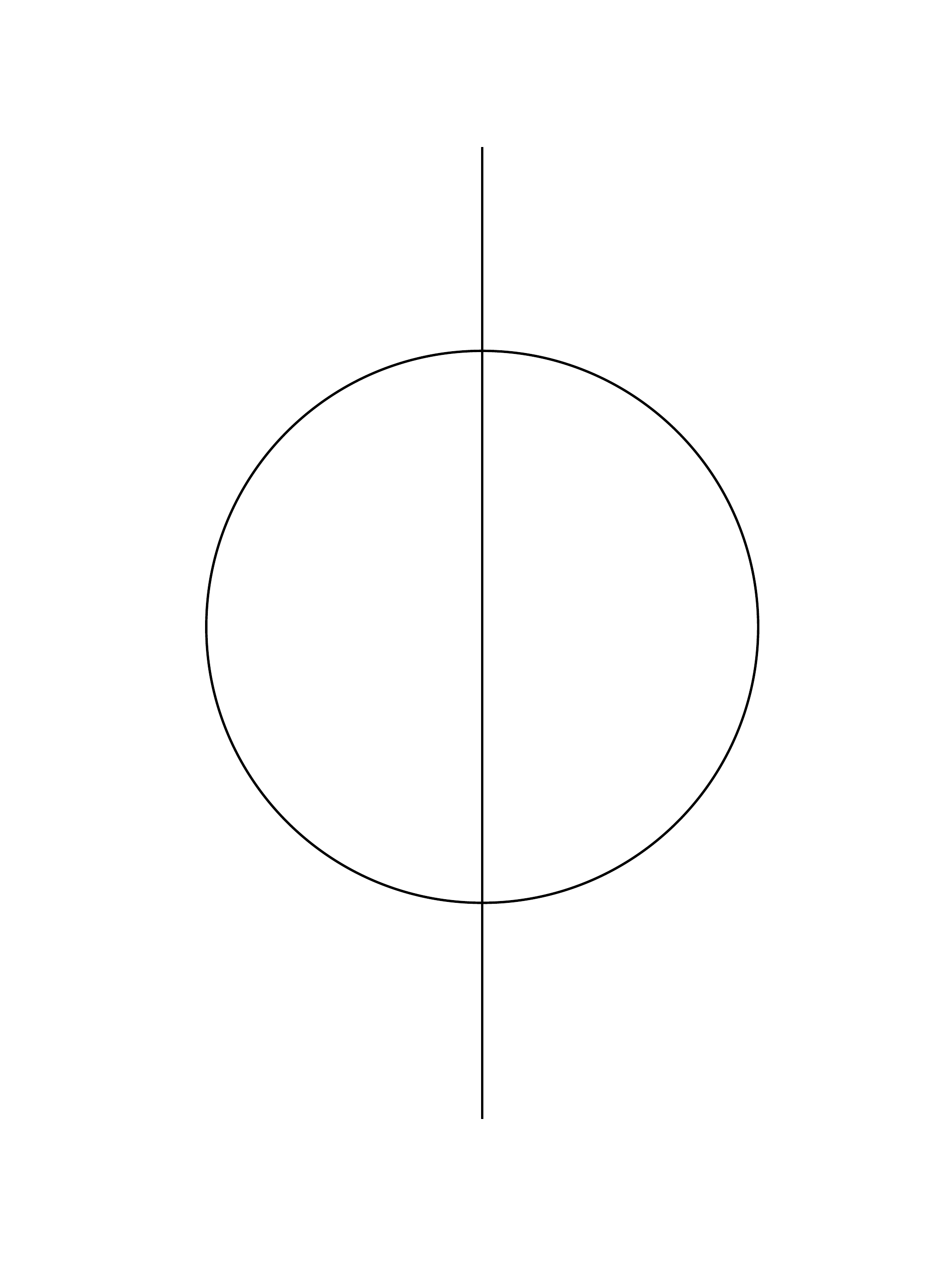}}
  \subfloat[]{\label{fig:b} \includegraphics[bb=20 0 230 500,
      clip,width=.27\textwidth]{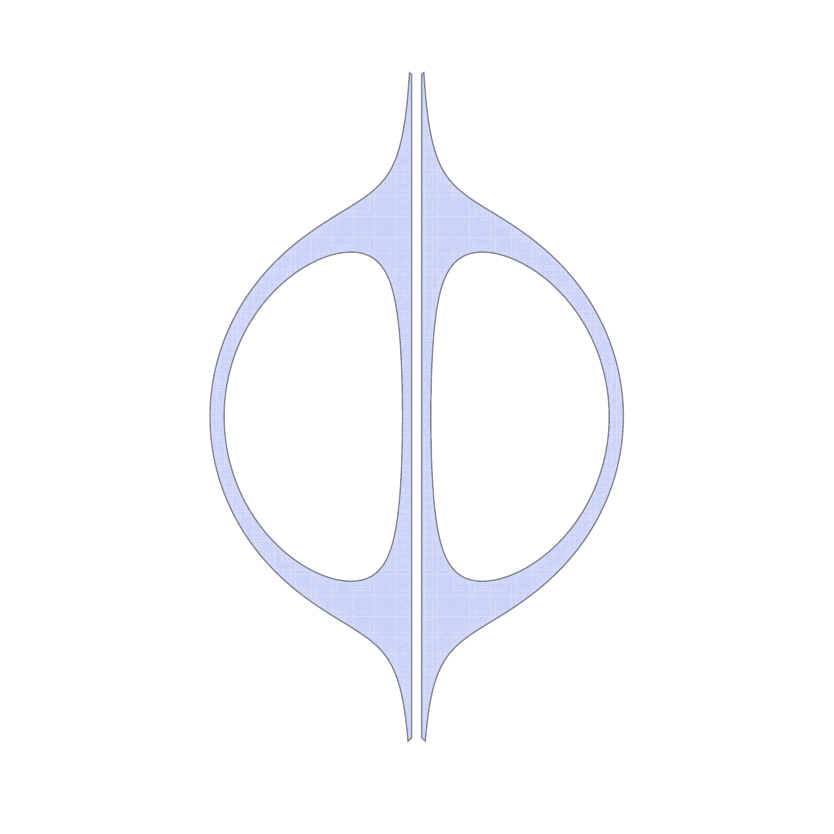}}
  \subfloat[]{\label{fig:c} \includegraphics[bb=80 40 500 700,clip,
      width=.18\textwidth]{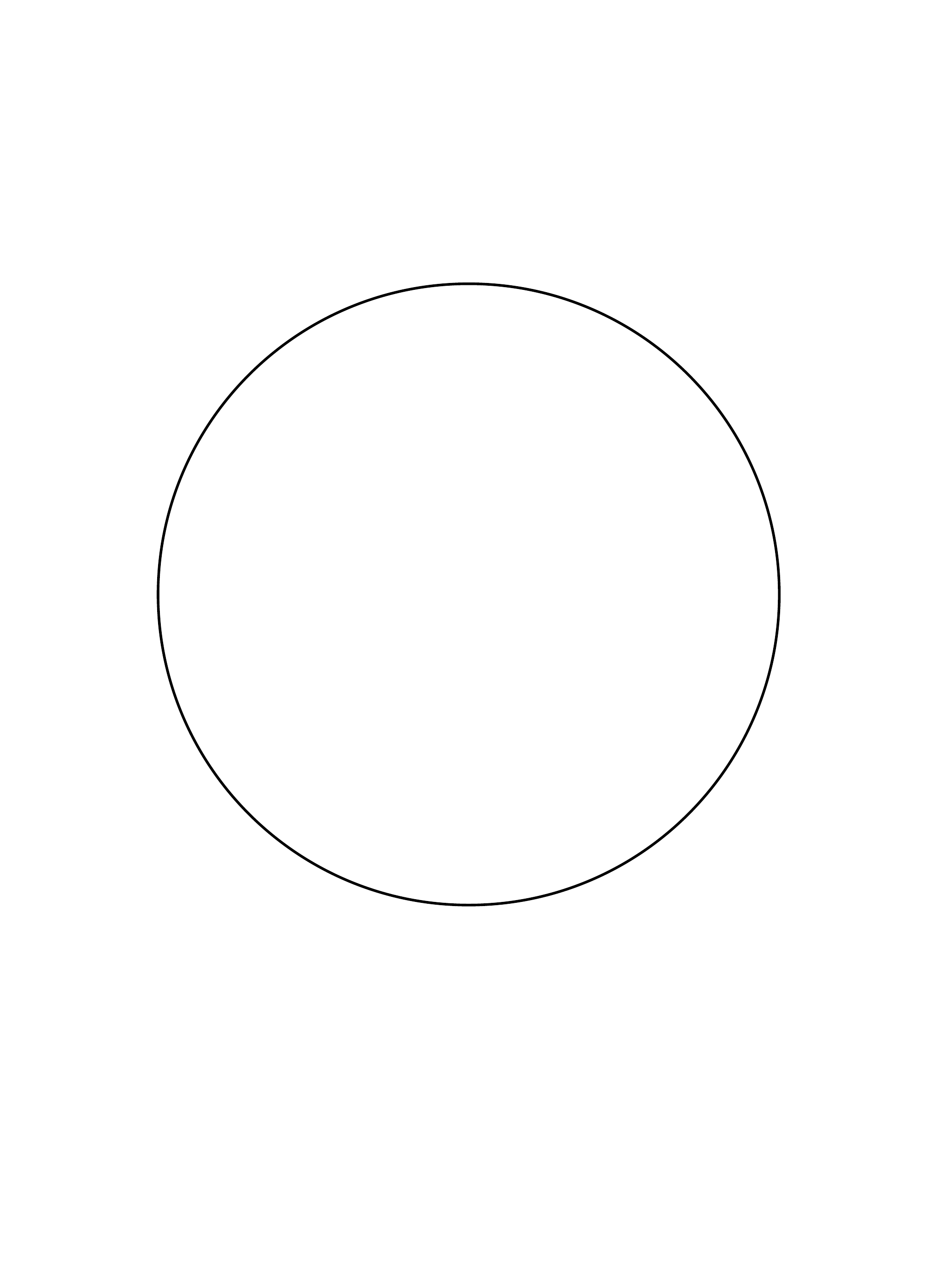}}
  \subfloat[]{\label{fig:d} \includegraphics[bb=40 0 200 500 ,clip,
      width=.22\textwidth]{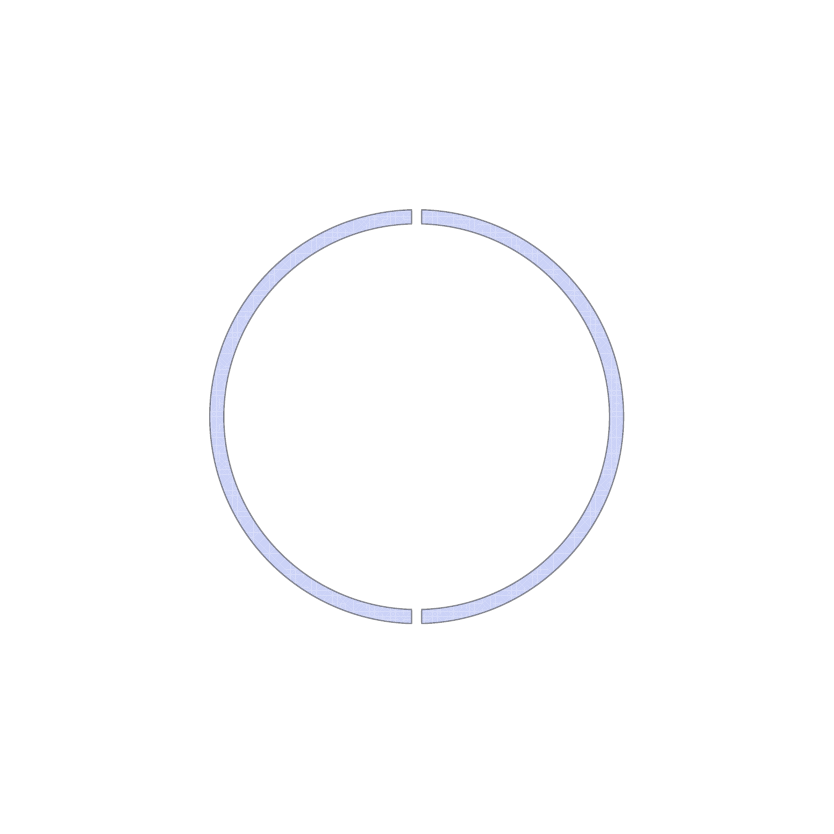}}
  \caption{Two examples.}
  \label{fig:example}
\end{figure}
\end{example}

We now prove Proposition \ref{prop:sectionfivemain}.

\proof[Proof of Proposition \ref{prop:sectionfivemain}]
We show both inclusions.  First let $\mathcal{\textbf{x}}\in \TT_{\limit}$, and we
show that $F(\mathcal{\textbf{x}})=0$.
In particular, we prove that $0\leq F^2(\mathcal{\textbf{x}})<\ep$ for every $\ep>0$.

Let $\ep>0$.
Since $F^2$ is
continuous,
there exists $\delta>0$ such that
\begin{equation}{\label{eqn:sectfiveCONTF}}|\mathcal{\textbf{x}}-\textbf{y}|^2<\delta \quad
\implies \quad |F^2(\mathcal{\textbf{x}})-F^2(\textbf{y})|<\fraction{\ep}{2}.\end{equation} After
possibly making $\delta$ smaller we can suppose that
$\delta<\fraction{\ep^2}{4}$.

From the definition of $\TT_{\limit}$
(cf. Notation \ref{not:limit}),
we have that
\begin{equation}{\label{eqn:Tzero}} \TT_{\limit} = \{ \mathcal{\textbf{x}} \mid \quad (\forall
\delta)(\delta>0\implies (\exists t)(\exists \textbf{y})(\textbf{y}\in \TT_t \wedge
|\mathcal{\textbf{x}}-\textbf{y}|^2+t^2<\delta))\}.\end{equation}

Since
$\mathcal{\textbf{x}}\in \TT_{\limit}$,
there exists
$t\in \mathbb{R}_+$
and $\textbf{y}\in \TT_t$ such that $|\mathcal{\textbf{x}}-\textbf{y}|^2+t^2<\delta$, and in
particular both $|\mathcal{\textbf{x}}-\textbf{y}|^2<\delta$ and
$t^2<\delta<\fraction{\ep^2}{4}$.  The former inequality implies that
$|F^2(\mathcal{\textbf{x}})-F^2(\textbf{y})|<\fraction{\ep}{2}$.  The latter inequality implies
$t<\fraction{\ep}{2}$, and this together with $\textbf{y}\in \TT_t$
implies
$$\begin{aligned} &P^2(\textbf{y})\leq t (Q^2(\textbf{y})-t^N)\\ \implies &
F^2(\textbf{y})Q^2(\textbf{y})\leq t (Q^2(\textbf{y})-t^N) \\ \implies & 0\leq
F^2(\textbf{y})\leq t-\frac{t^{N+1}}{Q^2(\textbf{y})}<t \\ \implies & 0\leq
F^2(\textbf{y})<\fraction{\ep}{2}.
\end{aligned}$$
So, $F^2(\textbf{y})<\fraction{\ep}{2}$.  Finally, note that $|F^2(\mathcal{\textbf{x}})|\leq
|F^2(\mathcal{\textbf{x}}) - F^2(\textbf{y})| + |F^2(\textbf{y})| < \fraction{\ep}{2} +
\fraction{\ep}{2} = \ep$.

We next prove the other inclusion, namely we show
$\Zer(F,\mathbb{R}^k)\cap \overline{B_k(0,R)}\subseteq \TT_{\limit}$.
Let $\mathcal{\textbf{x}}\in \Zer(F,\mathbb{R}^k)\cap \overline{B_k(0,R)}$.
We fix $\delta>0$ and  show that there exists
$t\in \mathbb{R}_+$
and $\textbf{y}\in \TT_t$
such that $|\mathcal{\textbf{x}}-\textbf{y}|^2+t^2<\delta$ (cf. Equation \ref{eqn:Tzero}).

There are two cases to consider.
\begin{itemize}
\item[$Q(\mathcal{\textbf{x}})\neq 0$:]
Since $Q(\mathcal{\textbf{x}})\neq 0$,
there exists $t>0$ such that $Q^2(\mathcal{\textbf{x}})\geq t^N$ and $t^2<\delta$.
Now, $\mathcal{\textbf{x}}\in \TT_t$ and $$|\mathcal{\textbf{x}}-\mathcal{\textbf{x}}|^2+t^2 \ = \ t^2 \ <\delta,$$ so
setting $\textbf{y}=\mathcal{\textbf{x}}$ we see that $\textbf{y}\in \TT_t$ and $|\mathcal{\textbf{x}}-\textbf{y}|+t^2<\delta$.
Thus, $\mathcal{\textbf{x}}\in \TT_{\limit}$ as desired.

\item[$Q(\mathcal{\textbf{x}})=0$:]

Let $\textbf{v}\in \mathbb{R}^k$ be generic, and
denote
$\widehat{P}(U)=P(\mathcal{\textbf{x}}+U\textbf{v})$,
$\widehat{Q}(U)=Q(\mathcal{\textbf{x}}+U\textbf{v})$, and
$\widehat{F}(U)=F(\mathcal{\textbf{x}}+U\textbf{v})$.
Note that
\begin{eqnarray}
\label{eqn:multiplicity}
\widehat{P} = \widehat{F}\widehat{Q}, \nonumber \\
\widehat{P}(0) = \widehat{Q}(0) = \widehat{F}(0) = 0.
\end{eqnarray}

If $F$ is not the zero polynomial, then neither is $\widehat{P}$, since $\textbf{v}$ is generic.
Indeed,
assume $F$ is not identically zero, and hence $P$
is not identically zero.
In order to prove that $\widehat{P}$ is not
identically zero for a generic choice of
$\textbf{v}$,
write $P = \sum_{0 \leq i \leq d} P_i$ where $P_i$ is the homogeneous
part of $P$ of degree $i$, and $P_d$ not identically zero. Then, it is
easy to see that $\widehat{P}(U) = P_d(\textbf{v}) U^d + \mathrm{\mbox{lower
    degree terms}}$.  Since $\mathbb{R}$ is an infinite field, a generic
choice of $\textbf{v}$ will avoid the set of zeros of $P_d$,
and thus,
$\widehat{P}$ is not identically zero.

We further require
that $\mathcal{\textbf{x}}+t\textbf{v}\in B_k(0,R)$ for $t>0$ sufficiently small.  For
generic $\textbf{v}$, this is true for either $\textbf{v}$ or $-\textbf{v}$, and so
after possibly replacing $\textbf{v}$ by $-\textbf{v}$
(and noticing that since $P_d$ is homogeneous we have $P_d(\textbf{v}) =
(-1)^d P_d(-\textbf{v})$) we may assume $\mathcal{\textbf{x}}+t\textbf{v}\in B_k(0,R)$ for $t>0$
sufficiently small.
Let $t_0>0$ be such that $\mathcal{\textbf{x}}+t\textbf{v}\in B_k(0,R)$ for $0<t< t_0$.

Denoting by $\nu = \mathrm{mult}_0(\hatP)$ and
$\mu = \mathrm{mult}_0(\hatQ)$, we have from (\ref{eqn:multiplicity})
that $\nu > \mu$.
Let
$$\begin{aligned} \hatP(U)&=\sum_{i=\nu}^{\deg_U \hatP} c_iU^i = U^\nu
  \cdot \sum_{i=0}^{\deg_U \hatP - \nu} c_{\nu+i}U^i = c_\nu U^\nu +
  \text{ (higher order terms)}, \\ \hatQ(U)&=\sum_{i=\mu}^{\deg_U \hatQ} d_iU^i = U^\mu \cdot
  \sum_{i=0}^{\deg_U \hatQ - \mu} d_{\mu+i}U^i = d_\mu U^\mu +
  \text{ (higher order terms)} \end{aligned}$$
where $c_\nu,d_\mu \neq 0$.

Then we have $$\begin{aligned} \hatP^2(U)&=c_\nu^2 U^{2\nu} +
  \text{ (higher order terms)}, \\
  \hatQ^2(U)&=d_\mu^2 U^{2\mu} +
  \text{ (higher order terms)}, \\
  D(U):=U(\hatQ^2(U)-U^N)&= U(d_\mu^2 U^{2\mu} +
  \text{ (higher order terms)}-U^N), \\
  D(U)-\hatP^2(U)&= d_\mu^2 U^{2\mu+1} + \text{ (higher order terms)} -
      U^{N+1}.
\end{aligned}$$

Since $\mu \leq \deg(Q)$ and $N = 2 \deg(Q) + 1$,
we have that $2\mu+1 < N+1$. Hence,
there exists $t_1 \in \mathbb{R}_+$ such that for
each $t$, $0 < t < t_1$,
we have $D(t)-\hatP^2(t)\geq 0$.
Thus,
$\mathcal{\textbf{x}}+t\textbf{v}\in \TT_t$ for each $t$, $0<t<\min\{t_0,t_1\}$.
Let $t_2 = (\frac{\delta}{|\textbf{v}|^2+1})^{1/2}$ and note that
for all $t$, $0<t< t_2$, we have $(|\textbf{v}|^2+1)t^2<\delta$.
Finally, if $t$ satisfies
$0<t< \min\{t_0,t_1,t_2\}$ then
$\mathcal{\textbf{x}}+t\textbf{v}\in \TT_t$,
and
$$|\mathcal{\textbf{x}}-(\mathcal{\textbf{x}}+t\textbf{v})|^2+t^2\ = \ (|\textbf{v}|^2+1)t^2 <
\delta.$$
Hence, setting $\mathcal{\textbf{y}}=\mathcal{\textbf{x}}+t\textbf{v}$ (cf. Equation \ref{eqn:Tzero}) we have shown that $\mathcal{\textbf{x}}\in \TT_\limit$ as desired.

The case where $F$ is the zero polynomial is straightforward.

\end{itemize}
\zz

\proof[Proof of Theorem \ref{thm:algebraic}]
For each $F\in \mathbb{R}[X_1,\ldots,X_k]$, by the conical structure at infinity of semi-algebraic sets
(see for instance \cite[page 188]{BPRbook2}), we have
that
there exists $R_F\in \mathbb{R}_+$ such that, for every $R>R_F$, the
semi-algebraic sets $\Zer(F,\mathbb{R}^k)\cap \overline{B_k(0,R)}$ and $\Zer(F,\mathbb{R}^k)$
are semi-algebraically
homeomorphic.

Let $\ell\in \mathbb{N}$, $F_1,\ldots,F_\ell\in \mathbb{R}[X_1,\ldots,X_k]$ such that each $F_i$ has additive complexity at most $a$ and, for every $F$ having additive complexity at most $a$, the algebraic sets $\Zer(F,\mathbb{R}^k),\Zer(F_i,\mathbb{R}^k)$ are semi-algebraically homeomorphic for some $i$, $1\leq i \leq \ell$
(see, for example \cite[Theorem 3.5]{Dries}).  Let $R=\max_{1\leq i \leq \ell} \{R_{F_i}\}$.

Let $F\in \{F_i\}_{1\leq i \leq \ell}$.
By Lemma \ref{lem:equivalence} there exists polynomials $P,Q \in \mathbb{R}[X_1,\ldots,X_k]$ such that
$FQ=P$, and such that $P,Q$
satisfies $P^2-T(Q^2-T^N)\in \mathbb{R}[X_1,\ldots,X_k,T]$ has division-free additive complexity bounded by $a+2$.
Let
$$
\TT=\{(\mathcal{\textbf{x}},t)\in
\mathbb{R}^k \times \mathbb{R}_+| \;
P^2(\mathcal{\textbf{x}}) \leq t(Q^2(\mathcal{\textbf{x}})-t^N) \wedge 
|\mathcal{\textbf{x}}|^2
\leq R^2
\}.
$$

By Proposition \ref{prop:sectionfivemain}
 we have that
$
\TT_{\limit} =
\Zer(F,\mathbb{R}^k)
\cap \overline{B_k(0,R)}.
$
Note the one-parameter semi-algebraic
family $\TT$ (where the last co-ordinate is the parameter) is
described by a formula having division-free
additive format
$(a+k+2,k+1)$.

By Theorem \ref{thm:main_weak}
we obtain a collection of
semi-algebraic sets $\mathcal{S}_{k,a+k+2}$
such that $\TT_{\limit}$, and hence $\Zer(F,\mathbb{R}^k)$, is homotopy equivalent
to some $S\in \mathcal{S}_{k,a+k+2}$
and $\#\mathcal{S}_{k,a+k+2} = 2^{O(k(k^2+a))^8}$,
which proves the theorem.

\zz
\end{subsection}

\begin{subsection}{The semi-algebraic case}
\label{subsec:semi-algebraic}
We first prove a generalization of Proposition \ref{prop:sectionfivemain}.

\begin{notation} Let $\X=(X_1,\dots,X_k)$ be a block of variables
and $\textbf{k}=(k_1,\dots,k_n)\in \mathbb{N}^n$ with
$\sum_{i=1}^n k_i=k$.  Let $\textbf{r}=(r_1,\dots,r_n)\in \mathbb{R}^n$ with $r_i>0$,
$i=1,\dots,n$.  Let $B_{\textbf{k}}(0,\textbf{r})$ denote the product
$$B_{\textbf{k}}(0,\textbf{r}):= B_{k_1}(0,r_1)\times \dots \times B_{k_n}(0,r_n).$$
\end{notation}

\begin{proposition}
\label{prop:level1}
Let $F_1,\dots,F_s,P_1,\dots,P_s,Q_1,\dots,Q_s\in
\mathbb{R}[\X^1,\dots,\X^n]$, $\mathcal{P}=\{F_1,\dots,F_s\}$
such that $F_iQ_i=P_i$, for all $i=1,\ldots,s$.  Suppose
$\X^i=(X^i_1,\dots,X^i_{k_i})$ and let \emph{$\textbf{k}=(k_1,\dots,k_n)$}.
Suppose $\phi$ is a $\mathcal{P}$-formula containing no negations and no
inequalities.  Let $$\begin{aligned} \bar{P_i}:=&P_i\prod_{j\neq
    i}Q_j, \\ \bar{Q}:=&\prod_j Q_j, \end{aligned}$$ and let
$\bar{\phi}$ denote the formula
obtained
from $\phi$
by replacing each $F_i=0$ with
$$\bar{P_i}^2 - U(\bar{Q}^2-U^N) \leq 0,$$
where $U$ is the last variable of $\bar{\phi}$,
$N=2\deg (\bar{Q})+1$.  Then,
for every $\emph{\textbf{r}}=(r_1,\dots,r_n)\in \mathbb{R}_+^n$,
we have (cf. Notation \ref{not:reali} and Notation \ref{not:limit})

\begin{equation}
\label{eqn:level1}
\Reali\left(\bigwedge_{i=1}^n (|\X^i|^2\leq r_i^2) \wedge \bar{\phi}
\wedge U>0
\right)_{\limit}=\Reali(\phi) \cap \overline{B_{\emph{\textbf{k}}}(0,\emph{\textbf{r}})}.
\end{equation}
\end{proposition}

\proof We follow the proof of Proposition \ref{prop:sectionfivemain}.
The only case which is not immediate is the case $\mathcal{\textbf{x}}\in
\Reali(\phi)\cap \overline{B_{\textbf{k}}(0,\textbf{r})}$ and $\bar{Q}(\mathcal{\textbf{x}})=0$.

\newcommand{\barphi}{\bar{\phi}}
\newcommand{\barphialpha}{{\bar{\phi}_{\alpha}}}

Suppose $\mathcal{\textbf{x}}\in \Reali(\phi)\cap \overline{B_{\textbf{k}}(0,\textbf{r})}$ and
that $\bar{Q}(\mathcal{\textbf{x}})=0$.  Since $\phi$ is a formula containing no
negations and no inequalities, it consists of conjunctions and
disjunctions of equalities.  Without loss of generality we can assume
that $\phi$ is written as a disjunction of conjunctions, and still
without negations.  Let
$$\phi = \bigvee_{\alpha} \phi_\alpha$$ where $\phi_\alpha$ is a
conjunction of equations.  As above let $\barphialpha$ be the formula
obtained from $\phi_\alpha$ after replacing each $F_i=0$ in
$\phi_\alpha$ with
$$\bar{P_i}^2\leq U(\bar{Q}^2-U^N),$$ $N=2\deg (\bar{Q})+1$.

We have
$$\begin{aligned} \Reali\left(\bigwedge_{i=1}^p (|\X^i|^2\leq r_i^2
  )\wedge \barphi
\wedge U>0
 \right)_{\limit} & =\Reali\left( \bigwedge_{i=1}^p
  (|\X^i|^2\leq r_i^2 )\wedge \left(\bigvee_\alpha \barphialpha \right)
\wedge U>0
  \right)_{\limit} \\ & =\Reali\left( \bigvee_{\alpha} \bigwedge_{i=1}^p
  (|\X^i|^2\leq r_i^2 )\wedge \barphialpha
\wedge U>0
\right)_{\limit} \\ & =
  \bigcup_\alpha \Reali\left(\bigwedge_{i=1}^p (|\X^i|^2\leq r_i^2 )\wedge
  \barphialpha
\wedge U>0
\right)_{\limit}.
\end{aligned}$$

In order to show that $\mathcal{\textbf{x}}\in
\Reali\left(\bigwedge_{i=1}^p (|\X^i|^2\leq r_i^2)\wedge \barphi
\wedge U>0
\right)_{\limit}$ it now suffices to show that if $\mathcal{\textbf{x}} \in
\Reali(\phi_\alpha)\cap \overline{B_{\textbf{k}}(0,\textbf{r})}$ and
$\bar{Q}(\mathcal{\textbf{x}})=0$, then $\mathcal{\textbf{x}} $
belongs to $ \Reali\left(\bigwedge_{i=1}^p
(|\X^i|^2\leq r_i^2 )\wedge \barphialpha
\wedge U>0
\right)_{\limit}$.

Let $\mathcal{\textbf{x}}\in \Reali(\phi_\alpha)\cap \overline{B_{\textbf{k}}(0,\textbf{r})}$ and suppose
$\bar{Q}(\mathcal{\textbf{x}})=0$. Let $\mathcal{Q}\subseteq \mathcal{P}$
consist of the polynomials of $\mathcal{P}$ appearing
in $\phi_\alpha$.
Let $\textbf{v}\in \mathbb{R}^k$ be generic, and
denote
$\widehat{P_i}(U)=\bar{P_i}(\mathcal{\textbf{x}}+U\textbf{v})$,
$\widehat{Q}(U)=\bar{Q}(\mathcal{\textbf{x}}+U\textbf{v})$, and
$\widehat{F_i}(U)=\bar{F}(\mathcal{\textbf{x}}+U\textbf{v})$.
Note that
\begin{eqnarray}
\label{eqn:multiplicity-two}
\widehat{P_i} = \widehat{F_i}\widehat{Q}, \nonumber \\
\widehat{P_i}(0) = \widehat{Q}(0) = \widehat{F_i}(0) = 0.
\end{eqnarray}

\newcommand{\hatPi}{\widehat{P_i}} As in the proof of Proposition
\ref{prop:sectionfivemain}, if $F_i\in \mathcal{Q}$ is not the zero
polynomial then
$\widehat{P_i}$ is not
identically zero.  Since $\phi_\alpha$ consists of a conjunction of
equalities and
$$\bigwedge_{F\in \mathcal{Q} \atop F\not\equiv 0} F=0 \iff \bigwedge_{F\in
  \mathcal{Q}} F=0,$$ we may assume that $\mathcal{Q}$ does not contain the zero
polynomial.  Under this assumption,
we have that for every $F_i\in \mathcal{Q}$ the univariate polynomial
$\widehat{P_i}$ is not identically zero.  As in the proof of
Proposition \ref{prop:sectionfivemain},
there exists $t_0\in \mathbb{R}_+$ such
that for all $t$, $0<t<t_0$, we have
$\mathcal{\textbf{x}}+t\textbf{v}\in B_{\textbf{k}}(0,\textbf{r})$.
Denoting by $\nu_i = \mathrm{mult}_0(\hatPi)$ and $\mu =
\mathrm{mult}_0(\hatQ)$, we have from (\ref{eqn:multiplicity-two}) that
$\nu_i > \mu$ for all $i=1,\dots,s$.

Let
$$\begin{aligned} \hatPi(U)&=\sum_{j=\nu_i}^{\deg_U \hatPi} c_jU^j =
  U^{\nu_i} \cdot \sum_{j=0}^{\deg_U \hatPi - \nu_i} c_{\nu_i+j}U^j =
  c_{\nu_i} U^{\nu_i} + \text{ (higher order terms)},
  \\ \hatQ(U)&=\sum_{j=\mu}^{\deg_U \hatQ} d_j^j = U^\mu \cdot
  \sum_{j=0}^{\deg_U \hatQ - \mu} d_{\mu+j}U^j = d_\mu U^\mu + \text{
    (higher order terms)} \end{aligned}$$ where $d_\mu \neq 0$ and
$c_{\nu_i}\neq 0 $.

Then we have $$\begin{aligned} \hatPi^2(U)&=c_{\nu_i}^2 U^{2\nu_i} +
  \text{ (higher order terms)}, \\
  \hatQ^2(U)&=d_\mu^2 U^{2\mu} +
  \text{ (higher order terms)}, \\
  D(t):=U(\hatQ^2(U)-U^N)&= U(d_\mu^2 U^{2\mu} +
  \text{ (higher order terms)}-U^N), \\
  D(U)-\hatPi^2(U)&= d_\mu^2 U^{2\mu+1} + \text{ (higher order terms)} -
      U^{N+1}.
\end{aligned}$$

Since $\mu \leq \deg(\bar{Q})$ and $N = 2 \deg(\bar{Q}) + 1$,
we have that $2\mu+1 < N+1$. Hence, there exists $t_{1,i}\in \mathbb{R}_+$ such that for all $t$, $0 < t < t_{1,i}$,
we have that $D(t)-\hatPi^2(t)\geq 0$, and thus
$\mathcal{\textbf{x}}+t\textbf{v}$ satisfies
$$\bar{P}_i^2(\mathcal{\textbf{x}}+t\textbf{v})\leq t(\bar{Q}^2(\mathcal{\textbf{x}}+t\textbf{v})-t^N).$$
Let $t_1=\min\{t_{1,1},\dots,t_{1,s}\}$.
Let $t_2 = (\frac{\delta}{|\textbf{v}|^2+1})^{1/2}$ and note that
for all $t\in \mathbb{R}$, $0<t< t_2$, we have $(|\textbf{v}|^2+1)t^2<\delta$.
Finally, if $t$ satisfies $0<t< \min\{t_0,t_1,t_2\}$ then
$$(\mathcal{\textbf{x}}+t\textbf{v},t)\in \Reali\left(\bigwedge_{i=1}^p
(|\X^i|^2\leq r_i^2)\wedge \barphialpha
\wedge U>0
\right)$$
and
$$|\mathcal{\textbf{x}}-(\mathcal{\textbf{x}}+t\textbf{v})|^2+t^2\ = \ (|\textbf{v}|^2)t^2 <
\delta,$$
and so we have shown that
$$\mathcal{\textbf{x}}\in \Reali\left(\bigwedge_{i=1}^p (|\X^i|^2\leq r_i^2)\wedge
\barphialpha
\wedge U>0
\right)_{\limit} .$$
\zz

Using the same notation as in Proposition \ref{prop:level1} above:
\begin{corollary}
\label{cor:level1}
Let $\phi$ be a $\mathcal{P}$-formula, containing no negations
and no inequalities, with 
$\mathcal{P} \subset \R[X_1,\ldots,X_k]$
with 
$\mathcal{P}\in \mathcal{A}_{k,a}$. 
Then, there exists a family
of polynomials $\mathcal{P}' \subset \R[X_1,\ldots,X_k,U]$, 
and a 
$\mathcal{P}'$-formula $\bar{\phi}$ 
satisfying (\ref{eqn:level1}), and such that  
$\mathcal{P}'\in \mathcal{A}^{\mathrm{div-free}}_{k+1,(k+a)(a+2)}$. 
\end{corollary}

\begin{proof}
The proof is immediate from 
Lemma \ref{lem:equivalence}, Remark \ref{rem:zero}, and 
the definition of $\bar{\phi}$.
\end{proof}

\begin{definition}
Let $\Phi$ be a $\mathcal{P}$-formula, $\mathcal{P}\subseteq \mathbb{R}[\X_1,\dots,\X_k]$, and say
that $\Phi$ is a \emph{$\mathcal{P}$-closed formula} if the formula $\Phi$ contains
no negations and all the inequalities in atoms of $\Phi$ are weak inequalities.
\end{definition}

Let $\mathcal{P} = \{F_1,\ldots,F_s\} \subset \mathbb{R}[X_1,\ldots,X_k]$, and
$\Phi$ a $\mathcal{P}$-closed formula.

For $R \in \mathbb{R}_+$, let $\Phi_R$ denote the formula $\Phi \wedge (|\X|^2 -R^2
\leq 0)$.

Let $\Phi^\dagger$ be the formula obtained from $\Phi$
by replacing each
occurrence of the atom $F_i\ast 0$, $\ast\in \{=,\leq,\geq\}$,
$i=1,\dots,s$, with
$$
\begin{aligned}F_i-V_i^2=0 & \mbox{ if } \ast \in \{\leq\}, \\
-F_i-V_i^2= 0 &  \mbox{ if } \ast \in\{\geq\}, \\
F_i=0 & \mbox{ if }\ast \in \{=\},
\end{aligned}
$$
and for
$R, R' \in \mathbb{R}_+$, let $\Phi^\dagger_{R,R'}$ denote the
formula
\[
\Phi^\dagger \wedge (U_1^2 + |\X|^2 -R^2 = 0) \wedge (U_2^2 +
|\V|^2 - R'^2 = 0).
\]

We have
\begin{proposition}{\label{prop:sectionfivemain2}}
\[
\Reali(\Phi) = \pi_{[1,k]}(\Reali(\Phi^\dagger)),
\]
and for all $0 < R \ll R'$,
\[
\Reali(\Phi_R) = \pi_{[1,k]}(\Reali(\Phi^\dagger_{R,R'})),
\]
\end{proposition}

\proof
Obvious.
\zz

Note that, for $0< R \ll R'$, $\pi_{[1,k]}|_{
\Reali(\Phi^\dagger_{R,R'})}$ is a continuous,
semi-algebraic surjection onto
$\Reali(\Phi_R)$.
Let $\pi_{R,R'}$ denote the map
$\pi_{[1,k]}|_{\Reali(\Phi^\dagger_{R,R'})}$.

\begin{proposition}
\label{prop:reduction2}
We have that  $\mathcal{J}^p_{\pi_{R,R'}}
(\Reali(\Phi^\dagger_{R,R'}))$
is $p$-equivalent to $\pi_{[1,k]}(\Reali(\Phi^\dagger_{R,R'}))$.
Moreover, for any two formulas $\Phi,\Psi$, the realizations
$\Reali(\Phi)$ and $\Reali(\Psi)$ are homotopy equivalent if,
for all $1 \ll R \ll R'$,
$$\Reali(\mathcal{J}^p_{\pi_{R,R'}}(\Phi^\dagger_{R,R'}))\simeq
\Reali(\mathcal{J}^p_{\pi_{R,R'}}(\Psi^\dagger_{R,R'}))$$
are homotopy equivalent for some $p > k$.
\end{proposition}
\proof
Immediate from Proposition \ref{prop:5} and
Propositions \ref{prop:top_basic} and \ref{prop:sectionfivemain2}.
\zz

Suppose that 
$\Phi$ has additive format 
bounded by $(a,k)$, and suppose that the number of polynomials
appearing $\Phi$ is $s$, 
and without loss of generality
we can assume that $s \leq k+a$ (see 
Remark \ref{rem:zero}). 
Then the sum of the
additive complexities of the polynomials appearing in 
$\Phi^{\dagger}_{R,R'}$ is bounded by 
$3a+3s+2\leq 3a+3(a+k)+2\leq 6(k+a)$, 
and the formula $\Phi^{\dagger}_{R,R'}$ has additive
format bounded by 
$(6(k+a),2k+a+2)$.  

Consequently, the additive format of the formula 
$$\Theta_1\wedge \Theta_2^{\Phi^\dagger_{R,R'}} \wedge \Theta_3^{\pi_{R,R'}}$$ is bounded by $(M,N)$,
$$
\begin{aligned}
M &= (p+1)(6k+6a+1)+\textstyle{\binom{p+1}{2}}(4k+2a+3) \\
N &= (p+1)(2k+a+3)+\textstyle{\binom{p+1}{2}}. 
\end{aligned}
$$
In the above, the estimates of Proposition \ref{prop:calDp} suffice, with $(a,k)$ replaced by $(6(k+a), 2k+a+2)$.  Now, applying Corollary \ref{cor:level1} we have that there exists a $\mathcal{P}'$-formula 
$$\overline{\left(\Theta_1\wedge \Theta_2^{\Phi^\dagger_{R,R'}} \wedge \Theta_3^{\pi_{R,R'}}\right)}$$
which satisfies Equation \ref{eqn:level1} and such that the \emph{division-free} additive format of this formula is bounded by 
$((N+M)(M+2) ,N+1 )$.  Finally, let $\mathcal{J}^{p}_{\pi_{R,R'}}(\Phi^\dagger_{R,R'})^\star$ denote
the formula,
with last variable $U$,
\begin{equation}
{\label{eqn:star}}
\Omega^R \wedge \overline{\left(\Theta_1 \wedge \Theta_2^{\Phi^\dagger_{R,R'}}\wedge \Theta_3^{\pi_{R,R'}}\right)}
\wedge U>0, 
\end{equation}
and we have that the \emph{division-free} additive format of $\mathcal{J}^p_{\pi_{R,R'}}(\Phi^\dagger_{R,R'})^\star$ is bounded by $(M',N+1)$, 
$$M'=(p+1)(2k+a+3)+(N+M)(M+2).$$
Note that $M'\leq 5M^2$.

We have shown the following, 

\begin{proposition}
\label{prop:boundonT}
Suppose that the sum of the  additive complexities of
$F_i, 1\leq i \leq s$, is bounded by $a$. Then, the semi-algebraic set
$\Reali(\mathcal{J}^{p}_{\pi_{R,R'}}(\Phi^{\dagger}_{R,R'})^\star)$ 
can be defined by a
$\mathcal{P}'$-formula with
$\mathcal{P}' \in
\mathcal{A}^{\mathrm{div-free}}_{5M^2,N+1}$, 
$$\begin{aligned}
M&= (p+1)(6k+6a+1)+2\textstyle{\binom{p+1}{2}}(4k+2a+3)\\
N&=  (p+1)(2k+a+3)+\textstyle{\binom{p+1}{2}}.
\end{aligned}$$
\end{proposition}

Finally, we obtain

\begin{proposition}
\label{prop:closedandbounded}
The number of distinct homotopy types of 
semi-algebraic subsets of
$\mathbb{R}^k$ defined by $\mathcal{P}$-closed formulas
with $\mathcal{P} \in \mathcal{A}_{a,k}$ is
bounded by
$2^{(k(k+a))^{O(1)}}$.
\end{proposition}
\begin{proof}
Let $\mathcal{P} \in \mathcal{A}_{a,k}$.
By the conical structure at infinity of semi-algebraic sets
(see, for instance \cite[page 188]{BPRbook2})
there exists
$R_{\mathcal{P}}>0$
such that, for all $R>R_\mathcal{P}$
and every $\mathcal{P}$-closed formula $\Phi$,
the
semi-algebraic
sets $\Reali(\Phi_R), \Reali(\Phi)$ are
semi-algebraically homeomorphic.

For each $a,k\in \mathbb{N}$, there are only finitely many semi-algebraic homeomorphism types
of semi-algebraic sets described by a $\mathcal{P}$-formula having additive complexity at most $(a,k)$ \cite[Theorem 3.5]{Dries}.
Let $\ell\in \mathbb{N}$, $\mathcal{P}_i\in \mathcal{A}_{a,k}$, and $\Phi_i$ a $\mathcal{P}_i$-formula, $1\leq i \leq \ell$, such that every semi-algebraic set described by a formula
of additive complexity at most $(a,k)$ is semi-algebraically homeomorphic to $\Reali(\Phi_i)$ for some $i$, $1\leq i \leq \ell$.
Let $R=\max_{1\leq i \leq \ell} \{R_{\mathcal{P}_i}\}$ and $R'\gg R$.

Let $\Phi\in \{\Phi_i\}_{1\leq i \leq \ell}$.
By Proposition \ref{prop:reduction2} it suffices to bound the number
of distinct
homotopy types of the semi-algebraic
set
$\Reali(\mathcal{J}^{k+1}_{\pi_{R,R'}}(\Phi^\dagger_{R,R'}))$.
By Proposition \ref{prop:level1}, we have that
$$\Reali\left(
\mathcal{J}^{k+1}_{\pi_{R,R'}}(\Phi^\dagger_{R,R'})^\star
 \right)_{\limit}
=\Reali(\mathcal{J}^{k+1}_{\pi_{R,R'}}(\Phi^\dagger_{R,R'})).$$
By Proposition \ref{prop:boundonT},
the division-free additive format of the formula
$\mathcal{J}^{k+1}_{\pi_{R,R'}}(\Phi^\dagger_{R,R'})^\star$
is bounded by
$(2M,N)$,
where $p=k+1$.
The proposition now follows immediately from
 Theorem \ref{thm:main_weak}.
\end{proof}

\proof[Proof of Theorem \ref{thm:main1}]
Using the construction of Gabrielov and Vorobjov
\cite{GV07} one can reduce
the case of arbitrary semi-algebraic sets to that of
a
closed and bounded one,
defined by a $\mathcal{P}$-closed formula,
without changing asymptotically the complexity estimates (see for example
\cite{BV06}).
The theorem then follows directly from Proposition \ref{prop:closedandbounded}
above.
\zz
\end{subsection}

\begin{subsection}{Proof of Theorem \ref{thm:main}}
\begin{proof}[Proof of Theorem \ref{thm:main}]
The proof is identical to that of the proof of Theorem \ref{thm:main_weak},
except that we use
Theorem \ref{thm:main1} instead of
Theorem \ref{thm:additive}.
\end{proof}

\end{subsection}

\end{section}

\bibliographystyle{plain}
\bibliography{master}

\end{document}